\documentclass{amsart}
\usepackage{amsmath,amssymb}
\usepackage[dvips]{graphics}
\usepackage[dvipdfmx]{graphicx}
\usepackage[all]{xy}

\newtheorem{Theorem}{Theorem}[section]
\newtheorem{Lemma}[Theorem]{Lemma}

\newtheorem{Corollary}[Theorem]{Corollary}

\theoremstyle{definition}

\newtheorem{Example}[Theorem]{Example}
\newtheorem{Problem}[Theorem]{Problem}
\newtheorem{Question}[Theorem]{Question}

\theoremstyle{remark}
\newtheorem{Remark}[Theorem]{Remark}

\def\labelenumi{\rm ({\makebox[0.8em][c]{\theenumi}})}
\def\theenumi{\arabic{enumi}}
\def\labelenumii{\rm {\makebox[0.8em][c]{(\theenumii)}}}
\def\theenumii{\roman{enumii}}

\renewcommand{\thefigure}
{\arabic{section}.\arabic{figure}}
\makeatletter
\@addtoreset{figure}{section}%{subsection}
\def\@thmcountersep{-}
\makeatother

\numberwithin{equation}{section}

\begin{document} 

\title[Generalization of the Conway-Gordon theorem]{Generalization of the Conway-Gordon theorem and intrinsic linking on complete graphs}

%    Information for first author
\author{Hiroko Morishita}
\address{Division of Mathematics, Graduate School of Science, Tokyo Woman's Christian University, 2-6-1 Zempukuji, Suginami-ku, Tokyo 167-8585, Japan}
\email{d17m104@cis.twcu.ac.jp}
%\thanks{}

%    Information for second author
\author{Ryo Nikkuni}
\address{Department of Mathematics, School of Arts and Sciences, Tokyo Woman's Christian University, 2-6-1 Zempukuji, Suginami-ku, Tokyo 167-8585, Japan}
\email{nick@lab.twcu.ac.jp}
\thanks{The second author was supported by JSPS KAKENHI Grant Numbers JP15K04881 and JP19K03500.}

%    General info

\date{}

\dedicatory{This article is dedicated to Professor Yoshiyuki Ohyama on his 60th birthday.}


\begin{abstract}
Conway and Gordon proved that for every spatial complete graph on six vertices, the sum of the linking numbers over all of the constituent two-component links is odd, and Kazakov and Korablev proved that for every spatial complete graph with arbitrary number of vertices greater than six, the sum of the linking numbers over all of the constituent two-component Hamiltonian links is even. In this paper, we show that for every spatial complete graph whose number of vertices is greater than six, the sum of the square of the linking numbers over all of the two-component Hamiltonian links is determined explicitly in terms of the sum over all of the triangle-triangle constituent links. As an application, we show that if the number of vertices is sufficiently large then every spatial complete graph contains a two-component Hamiltonian link whose absolute value of the linking number is arbitrary large. Some applications to rectilinear spatial complete graphs are also given. 
\end{abstract}

\maketitle

\section{Introduction} 

Throughout this paper we work in the piecewise linear category. Let $G$ be a finite graph. An embedding $f$ of $G$ into ${\mathbb R}^{3}$ is called a {\it spatial embedding} of $G$, and the image $f(G)$ is called a {\it spatial graph} of $G$. Two spatial embeddings $f$ and $g$ of $G$ are said to be {\it equivalent} if there exists a self homeomorphism $\Phi$ on ${\mathbb R}^{3}$ such that $\Phi(f(G)) = g(G)$. Let $\gamma$ be a subgraph of $G$ homeomorphic to the circle. We call $\gamma$ a {\it cycle} of $G$. For a cycle $\gamma$ (resp. a disjoint union of cycles $\lambda$) and a spatial embedding $f$ of $G$, we call $f(\gamma)$ (resp. $f(\lambda)$) a {\it constituent knot} (resp. a {\it constituent link}) of the spatial graph $f(G)$. In particular, we say that a constituent knot (or link) of a spatial graph $f(G)$ is {\it Hamiltonian} if it contains all vertices of $f(G)$. 

Let $K_{n}$ be the {\it complete graph} on $n$ vertices, that is the graph consisting of $n$ vertices $1,2,\ldots,n$ such that each pair of its distinct vertices $i$ and $j$ is connected by exactly one edge $\overline{ij}$. Then the following fact is well-known as the {\it Conway-Gordon theorems}. 

\begin{Theorem}\label{CG} {\rm (Conway-Gordon \cite{CG83})} 
\begin{enumerate}
\item For any spatial embedding $f$ of $K_{6}$, the sum of the linking numbers over all of the constituent $2$-component Hamiltonian links of $f(K_{6})$ is odd. 
\item For any spatial embedding $f$ of $K_{7}$, the sum of the second coefficients of the Conway polynomials over all of the constituent Hamiltonian knots of $f(K_{7})$ is odd. 
\end{enumerate}
\end{Theorem}

Theorem \ref{CG} also implies that every spatial graph of $K_{6}$ contains a nonsplittable $2$-component link, and every spatial graph of $K_{7}$ contains a nontrivial knot. Furthermore, the authors gave in \cite{MN19} a generalized Conway-Gordon theorem for $K_{n}$ with arbitrary number of vertices $n\ge 6$ as follows, where the case of $n = 6$ has been proven by the second author in \cite[Theorem 1.3]{Nikkuni09}.

\begin{Theorem}\label{mainthm} {\rm (Morishita-Nikkuni \cite[Theorem 1.3]{MN19})} 
Let $n\ge 6$ be an integer. For any spatial embedding $f$ of $K_{n}$, we have 
\begin{eqnarray*}\label{maintheorem}
&&\sum_{\gamma\in \Gamma_{n}(K_{n})}a_{2}(f(\gamma))
- (n-5)!\sum_{\gamma\in \Gamma_{5}(K_{n})}a_{2}(f(\gamma))\\
&=& 
\frac{(n-5)!}{2} 
\bigg(
\sum_{\lambda\in \Gamma_{3,3}(K_{n})}{{\rm lk}(f(\lambda))}^{2}
- \binom{n-1}{5}
\bigg),\nonumber 
\end{eqnarray*}
where {\rm lk} denotes the linking number and $a_{2}$ denotes the second coefficient of the Conway polynomial. 
\end{Theorem}

Moreover, it has been also shown in \cite{MN19} that for a spatial embedding $f$ of $K_{n}$ with $n\ge 7$, the sum of $a_{2}$ over all of the constituent Hamiltonian knots of $f(K_{n})$ modulo $(n-5)!$ does not depend on the choice of $f$ and the modulo $(n-5)!$ reduction has been given explicitly as follows. 

\begin{Corollary}\label{maincor0} {\rm (\cite[Corollary 1.5]{MN19})} 
Let $n\ge 7$ be an integer. For any spatial embedding $f$ of $K_{n}$, we have the following congruence modulo $(n-5)!$: 
\begin{eqnarray*}
\sum_{\gamma\in \Gamma_{n}(K_{n})}a_{2}(f(\gamma)) 
\equiv 
\left\{
   \begin{array}{@{\,}lll}
   {\displaystyle \frac{(n-5)!}{2}} & (n\equiv 0,7\pmod{8}) \\
   0 & (n\not\equiv 0,7\pmod{8}). 
   \end{array}
\right.
\end{eqnarray*}
\end{Corollary}

On the other hand, for linking numbers of constituent $2$-component Hamiltonian links of a spatial graph of $K_{n}$ with $n\ge 7$, the followings are known. 

\begin{Theorem}\label{KK} 
\begin{enumerate}
\item {\rm (Vesnin-Litvintseva \cite{YL10})} Let $n\ge 7$ be an integer. For any spatial embedding $f$ of $K_{n}$, there exists a constituent $2$-component Hamiltonian link of $f(K_{n})$ with odd linking number. 
\item {\rm (Kazakov-Korablev \cite{KK14})} 
Let $n\ge 7$ be an integer. For any spatial embedding $f$ of $K_{n}$, the sum of the linking numbers over all of the constituent $2$-component Hamiltonian links of $f(K_{n})$ is even. 
\end{enumerate}
\end{Theorem}

Our purposes in this paper are to refine Theorem \ref{KK} (2) by giving its integral lift in terms of the square of the linking number and to investigate the behavior of the nonsplittable $2$-component Hamiltonian links in a spatial complete graph. In the following, we call a cycle of $G$ containing exactly $p$ edges a {\it $p$-cycle} of $G$ and denote the set of all $p$-cycles of $G$ by $\Gamma_{p}(G)$. Moreover, we denote the set of all pairs of two disjoint cycles of $G$ consisting of a $p$-cycle and a $q$-cycle by $\Gamma_{p,q}(G)$. For a spatial embedding $f$ of $G$ and an element $\lambda$ in $\Gamma_{p,q}(G)$, we also call $f(\lambda)$ a constituent $2$-component link {\it of type $(p,q)$}. Then we have the following, where the case of $p=4,\ q=3$ has already been obtained in \cite[Theorem 2.3]{MN19} (see also (\ref{k73433})). 

\begin{Theorem}\label{lkrefine} 
Let $n\ge 6$ be an integer and $p,q\ge 3$ two integers satisfying $n=p+q$. For any spatial embedding $f$ of $K_n$, we have 
\begin{eqnarray}\label{ilKK0}
\sum_{\lambda\in\Gamma_{p,q}(K_n)}{\rm lk}(f(\lambda))^2
=
\left\{
   \begin{array}{@{\,}lll}
   (n-6)!{\displaystyle\sum_{\lambda\in\Gamma_{3,3}(K_n)}{\rm lk}(f(\lambda))^2} & (p=q) \\
   2\cdot (n-6)!{\displaystyle\sum_{\lambda\in\Gamma_{3,3}(K_n)}{\rm lk}(f(\lambda))^2} & (p\neq q). 
   \end{array}
\right. 
\end{eqnarray}
In particular, we also have 
\begin{eqnarray}\label{ilKK}
\sum_{p+q=n}
\sum_{\lambda\in\Gamma_{p,q}(K_n)}{\rm lk}(f(\lambda))^2
=
(n-5)!\sum_{\lambda\in\Gamma_{3,3}(K_n)}{\rm lk}(f(\lambda))^2.
\end{eqnarray}
\end{Theorem}

Note that Theorem \ref{KK} (2) can be recovered by taking the modulo two reduction of (\ref{ilKK}), namely (\ref{ilKK}) gives an integral lift of Theorem \ref{KK} (2). Since $K_{n}$ contains $K_{6}$ if $n\ge 6$, by Theorem \ref{CG} (1) we have $\sum_{\lambda\in\Gamma_{3,3}(K_n)}{\rm lk}(f(\lambda))^2 > 0$. Therefore (\ref{ilKK0}) implies that every spatial graph of $K_{n}$ with $n\ge 6$ contains a constituent $2$-component Hamiltonian link of any type $(p,q)$ with nonzero linking number. In that sense, (\ref{ilKK0}) is also a refinement of Theorem \ref{KK} (1).

By Theorem \ref{lkrefine}, we obtain the following congruences as a corollary, which are remarkable generalizations of Theorem \ref{KK} (2).

\begin{Corollary}\label{maincor1} 
Let $n\ge 6$ be an integer and $p,q\ge 3$ two integers satisfying $n=p+q$. Let $f$ be a spatial embedding of $K_n$. 
\begin{enumerate}
\item If $p=q$, then we have the following congruence modulo $2\cdot (n-6)!$: 
\begin{eqnarray*}\label{congru1}
\sum_{\lambda\in\Gamma_{p,q}(K_n)}{\rm lk}(f(\lambda))^2
\equiv 
\left\{
   \begin{array}{@{\,}lll}
   (n-6)! & (n\equiv 6,7\pmod{8})\\
   0 & (n\not\equiv 6,7\pmod{8}). 
   \end{array}
\right. 
\end{eqnarray*}

\item If $p\neq q$, then we have the following congruence modulo $4\cdot (n-6)!$: 
\begin{eqnarray*}\label{congru2}
\sum_{\lambda\in\Gamma_{p,q}(K_n)}{\rm lk}(f(\lambda))^2
\equiv 
\left\{
   \begin{array}{@{\,}lll}
   2\cdot (n-6)! & (n\equiv 6,7\pmod{8})\\
   0 & (n\not\equiv 6,7\pmod{8}). 
   \end{array}
\right. 
\end{eqnarray*}

\item We have the following congruence modulo $2\cdot (n-5)!$: 
\begin{eqnarray*}\label{congru3}
\sum_{p+q=n}
\sum_{\lambda\in\Gamma_{p,q}(K_n)}{\rm lk}(f(\lambda))^2 
\equiv 
\left\{
   \begin{array}{@{\,}lll}
   (n-5)! & (n\equiv 6,7\pmod{8})\\
   0 & (n\not\equiv 6,7\pmod{8}). 
   \end{array}
\right. 
\end{eqnarray*}
\end{enumerate}
\end{Corollary}

Moreover, by Theorem \ref{lkrefine}, we also obtain the following inequalities as a corollary, where the case of $p=q=3$ has already been obtained from Theorem \ref{CG} (1), and the case of $p=3$, $q=4$ has also been observed in \cite[Theorem 1]{FM09}.

\begin{Corollary}\label{Cor9}
Let $n\ge 6$ be an integer and $p,q\ge 3$ two integers satisfying $n=p+q$. For any spatial embedding $f$ of $K_n$, we have 
\begin{eqnarray}\label{ieMN0}
\sum_{\lambda\in\Gamma_{p,q}(K_n)}{\rm lk}(f(\lambda))^2
\ge
\left\{
   \begin{array}{@{\,}lll}
   {\displaystyle \frac{n!}{6!}} & (p=q) \\
   {\displaystyle 2\cdot \frac{n!}{6!}} & (p\neq q). 
   \end{array}
\right. 
\end{eqnarray}

In particular, we have 
\begin{eqnarray}\label{ieMN}
\sum_{p+q=n}
\sum_{\lambda\in\Gamma_{p,q}(K_n)}{\rm lk}(f(\lambda))^2 
\ge
(n-5)\cdot \frac{n!}{6!}.
\end{eqnarray}
\end{Corollary}

The lower bound of each of the inequalities in Corollary \ref{Cor9} is sharp, see Remark \ref{ineq_sharp}. We also remark here that the minimum number of nonsplittable constituent $2$-component links of a spatial graph of $K_{n}$ with $n\le 9$ has been investigated by Fleming-Mellor \cite{FM09} and Abrams-Mellor-Trott \cite{AMT13}. It can be said that (\ref{ieMN}) gives an algebraic estimation from below of the number of nonsplittable constituent $2$-component Hamiltonian links of a spatial graph of $K_{n}$ with $n\ge 6$ (see also Remark \ref{minnumlk8}).

In addition, Corollary \ref{Cor9} also gives an information of the maximum value of the absolute value of the linking numbers over all of the $2$-component Hamiltonian links of type $(p,q)$ of a spatial graph of $K_{n}$ with $n\ge 6$.

\begin{Corollary}\label{maxlkcor}
Let $n\ge 6$ be an integer and $p,q\ge 3$ two integers satisfying $n=p+q$. For any spatial embedding $f$ of $K_n$, we have 
\begin{eqnarray*}\label{maxlk}
\mathop{\rm max}\limits_{\lambda\in \Gamma_{p,q}(K_{n})}
\left|{\rm lk}(f(\lambda))\right|
\ge
\frac{\sqrt{pq}}{3\sqrt{10}}. 
\end{eqnarray*}
\end{Corollary}

Corollary \ref{maxlkcor} says that if $n$ is sufficiently large then every spatial graph of $K_{n}$ contains a $2$-component Hamiltonian link of type $(p,q)$ whose absolute value of the linking number is arbitrary large. Actually we have the following. 

\begin{Corollary}\label{lkmaxham}
Let $p,q\ge 3$ be two integers. For a spatial embedding $f$ of $K_{p+q}$ and a positive integer $m$, if $pq > 90(m-1)^{2}$ then there exists a pair of disjoint cycles $\lambda\in \Gamma_{p,q}(K_{p+q})$ such that $|{\rm lk}(f(\lambda))|\ge m$. 
\end{Corollary}

It has already been known that for any positive integer $m$, there exists a positive integer $n$ such that every spatial graph of $K_{n}$ contains a $2$-component link $L$ with $|{\rm lk}(L)|\ge m$, see Flapan \cite{F02}, Shirai-Taniyama \cite{ST03}. In particular, Shirai-Taniyama showed in \cite{ST03} that for any spatial embedding $f$ of $K_{6\cdot 2^{k}}$, there exists a pair of disjoint cycles $\lambda\in \Gamma_{3\cdot 2^{k},3\cdot 2^{k}}(K_{6\cdot 2^{k}})$ such that $|{\rm lk}(f(\lambda))|\ge 2^{k}$. Since for every positive integer $m$, there exists an integer $k$ such that $2^{k-1}<m\le 2^{k}$, they also showed that $K_{12m}\supset K_{6\cdot 2^{k}}$ and therefore for any spatial embedding $f$ of $K_{12m}$, there exists a pair of disjoint cycles $\lambda\in \Gamma_{3\cdot 2^{k},3\cdot 2^{k}}(K_{12m})$ such that $|{\rm lk}(f(\lambda))|\ge 2^{k}\ge m$. The new knowledge obtained in Corollary \ref{lkmaxham} is the fact that for any positive integer $m$, there exists a positive integer $n$ such that every spatial graph of $K_{n}$ contains a $2$-component Hamiltonian link $L$ of any type $(p,q)$ with $|{\rm lk}(L)|\ge m$. 

The paper is organized as follows. In Section $2$, we prove Theorem \ref{lkrefine} in the case of $n=8$. In Section $3$, we prove Theorem \ref{lkrefine} in general case and Corollaries \ref{maincor1}, \ref{Cor9}, \ref{maxlkcor} and \ref{lkmaxham}. In Section $4$, we also mention some applications of our results to rectilinear spatial graphs, which are objects appearing in polymer chemistry as a mathematical model for chemical compounds.

\section{Proof of Theorem \ref{lkrefine}: $n=8$}%4.2

First we show Theorem \ref{lkrefine} in the case of $n=8$. Namely we prove the following.

\begin{Theorem}\label{lkthm}
For any spatial embedding $f$ of $K_8$, we have 
\begin{eqnarray}
&& \sum_{\lambda\in\Gamma_{4,4}(K_8)}{\rm lk}(f(\lambda))^2
= 2\sum_{\lambda\in\Gamma_{3,3}(K_8)}{\rm lk}(f(\lambda))^2,\label{4433}\\
&& \sum_{\lambda\in\Gamma_{3,5}(K_8)}{\rm lk}(f(\lambda))^2
= 4\sum_{\lambda\in\Gamma_{3,3}(K_8)}{\rm lk}(f(\lambda))^2.\label{3533}
\end{eqnarray}
\end{Theorem}

We show four lemmas which are needed to prove Theorem \ref{lkthm}. In the following, we introduce some subgraphs of $K_{n}$ with $n\ge 4$ which are used in these proofs (and $\S 3$). We denote the edge of $K_{n}$ connecting two distinct vertices $i$ and $j$ by $\overline{ij}$, and denote a path of length $2$ of $K_{n}$ consisting of two edges $\overline{ij}$ and $\overline{jk}$ by $\overline{ijk}$. We denote the subgraph of $K_{n}$ obtained from $K_{n}$ by deleting the vertex $m$ and all of the edges incident to $m$ by $K_{n-1}^{(m)}\ (m = 1,2,\ldots,n)$. Actually $K_{n-1}^{(m)}$ is isomorphic to $K_{n-1}$ for any $m$. For $1\le i<j\le n$ and $i,j\neq m$, let $F_{ij}^{(m)}$ be the subgraph of $K_{n}$ obtained from $K_{n}$ by deleting the edges $\overline{ij}$ and $\overline{mk}$ for all $k$ with $1\le k\le n,\ k\neq i,j$. Note that $F_{ij}^{(m)}$ is homeomorphic to $K_{n-1}$, namely $F_{ij}^{(m)}$ is obtained from $K_{n-1}^{(m)}$ by subdividing the edge $\overline{ij}$ by the vertex $m$, see Fig. \ref{Knsubdivide}.

\begin{figure}[htbp]
      \begin{center}
      \scalebox{0.425}{\includegraphics*{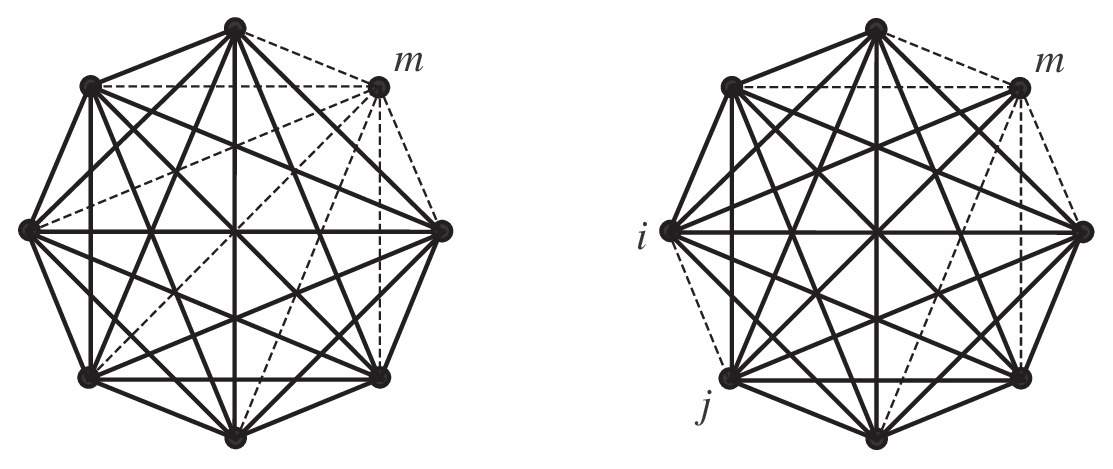}}
%\scalebox{0.65}{\includegraphics[bb = 39 283 573 508, width=12cm,clip]{Knsubdivide2.pdf}}
      \end{center}
   \caption{$K_{n-1}^{(m)}$ and $F_{ij}^{(m)}$ ($n = 8$)}
  \label{Knsubdivide}
\end{figure} 

Moreover, we denote a path of length $3$ that is not a $3$-cycle of $K_n$ consisting of three mutually distinct edges $\overline{ik}$, $\overline{kl}$ and $\overline{lj}$ by $\overline{iklj}$. 
For two distinct vertices $k$ and $l$ of $K_{n}$, 
we denote the subgraph of $K_n$ obtained from $K_n$ by deleting the vertices $k,l$ and all of the edges incident to $k,l$ by $K_{n-2}^{(kl)}\ (=K_{n-2}^{(lk)})$. Actually $K_{n-2}^{(kl)}$ is isomorphic to $K_{n-2}$ for any $k,l$. 
For two distinct vertices $i,j$ of $K_{n}$ with $i, j\neq k, l$, let $F_{ij}^{(kl)}$ be the subgraph of $K_n$ obtained from $K_n$ by deleting the edges $\overline{ij}$, $\overline{km}\ (1\leq m\leq n,\ m\neq i, l)$ and $\overline{lm'}\ (1\leq m'\leq n,\ m'\neq j, k)$, and 
$F_{ij}^{(lk)}$ the subgraph of $K_n$ obtained from $K_n$ by deleting the edges $\overline{ij}$, $\overline{km}\ (1\leq m\leq n,\ m\neq j, l)$ and $\overline{lm'}\ (1\leq m'\leq n,\ m'\neq i, k)$.  
Note that both $F_{ij}^{(kl)}$ and $F_{ij}^{(lk)}$ are obtained from $K_{n-2}^{(kl)}$ by subdividing the edge $\overline{ij}$ by the vertices $k,l$, see Fig \ref{Knsubdivide2}.

\begin{figure}[htbp]
      \begin{center}
            \scalebox{0.6}{\includegraphics*{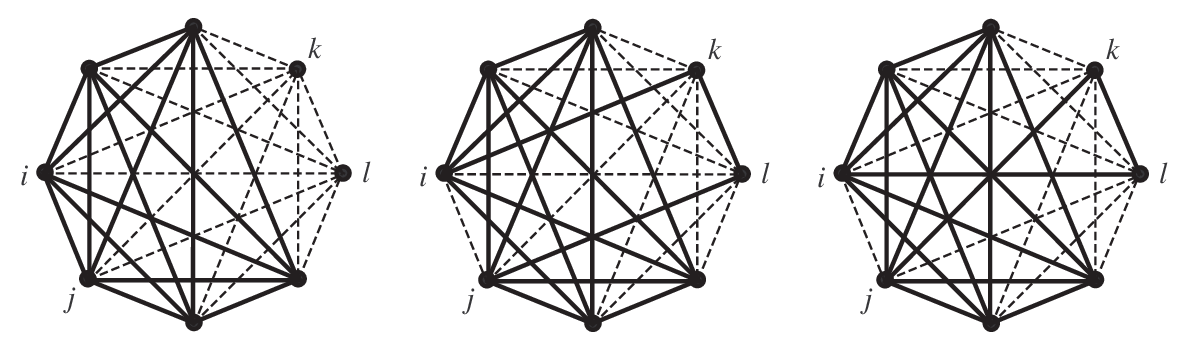}}
%\scalebox{0.8}{\includegraphics[bb = 21 312 591 480, width=14cm,clip]{Knsubdivide4.pdf}}
      \end{center}
   \caption{$K_{n-2}^{(kl)}$, $F_{ij}^{(kl)}$ and $F_{ij}^{(lk)}$ ($n = 8$)}
  \label{Knsubdivide2}
\end{figure} 

\begin{Lemma}\label{K8-48}
For any spatial embedding $f$ of $K_8$, we have 
\begin{eqnarray*}
\sum_{\gamma\in\Gamma_7(K_8)}a_2(f(\gamma))-6\sum_{\gamma\in\Gamma_5(K_8)}a_2(f(\gamma))
=2\sum_{\lambda\in\Gamma_{3,3}(K_8)}{\rm lk}(f(\lambda))^2-48.
\end{eqnarray*}
\end{Lemma}

\begin{proof}
Let $f$ be a spatial embedding of $K_8$. Then by applying Theorem \ref{mainthm} in the case of $n=7$ to the embedding $f$ restricted to $K_{7}^{(m)}\ (m=1,2,\ldots,8)$, we have 
\begin{eqnarray}\label{K8-480}
\sum_{\gamma\in\Gamma_7(K_{7}^{(m)})}a_2(f(\gamma))-2\sum_{\gamma\in\Gamma_5(K_{7}^{(m)})}a_2(f(\gamma))
=\sum_{\lambda\in\Gamma_{3,3}(K_{7}^{(m)})}{\rm lk}(f(\lambda))^2-6.
\end{eqnarray}
Let us take the sum of both sides of (\ref{K8-480}) for all $m$. 
%1
Since for each $t$-cycle $\gamma$ of $K_8$ is shared by exactly $(8-t)$ $K_{7}^{(m)}$'s, we have 
\begin{eqnarray}\label{K8-481} 
\sum_{m=1}^{8}\sum_{\gamma\in\Gamma_{t}(K_{7}^{(m)})}a_2(f(\gamma))=(8-t) \sum_{\gamma\in\Gamma_{t}(K_8)}a_2(f(\gamma)).
\end{eqnarray}
%3
On the other hand, we have 
\begin{eqnarray}\label{K8-483} 
\sum_{m=1}^{8}\sum_{\lambda\in\Gamma_{3,3}(K_{7}^{(m)})}{\rm lk}(f(\lambda))^{2}=2\sum_{\lambda\in\Gamma_{3,3}(K_8)}{\rm lk}(f(\lambda))^{2}.
\end{eqnarray}
By combining (\ref{K8-481}) and (\ref{K8-483}) with (\ref{K8-480}), we have the result. 
\end{proof}

\begin{Lemma}\label{413}
For any spatial embedding $f$ of $K_8$, we have 
\begin{eqnarray*}
&&16\sum_{\gamma\in\Gamma_8(K_8)}a_2(f(\gamma))-14\sum_{\gamma\in\Gamma_7(K_8)}a_2(f(\gamma))-12\sum_{\gamma\in\Gamma_5(K_8)}a_2(f(\gamma))\\
&=&5\sum_{\lambda\in\Gamma_{3,5}(K_8)}{\rm lk}(f(\lambda))^2-336.\nonumber
\end{eqnarray*}
\end{Lemma}

\begin{proof}
Let $f$ be a spatial embedding of $K_8$. Then by applying Theorem \ref{mainthm} in the case of $n=6$ to the embedding $f$ restricted to $F_{ij}^{(kl)}$ $(1\le i<j\le 8,\ i,j\neq k,l)$, we have 
\begin{eqnarray}\label{f78}
&&2\Bigg(
\sum_{\substack{\gamma\in\Gamma_8(F_{ij}^{(kl)})\\ \overline{iklj}\subset\gamma}}a_2(f(\gamma))+\sum_{\substack{\gamma\in\Gamma_6(F_{ij}^{(kl)})\\ \overline{iklj}\not\subset\gamma}}a_2(f(\gamma)) \\
&& -\sum_{\substack{\gamma\in\Gamma_7(F_{ij}^{(kl)})\\ \overline{iklj}\subset\gamma}}a_2(f(\gamma))-\sum_{\substack{\gamma\in\Gamma_5(F_{ij}^{(kl)})\\ \overline{iklj}\not\subset\gamma}}a_2(f(\gamma))
\Bigg)\nonumber\\
&=&\sum_{\substack{\lambda=\gamma\cup\gamma'\in\Gamma_{3,5}(F_{ij}^{(kl)})\\ \gamma\in\Gamma_5(F_{ij}^{(kl)}),\ \gamma'\in\Gamma_3(F_{ij}^{(kl)})\\ \overline{iklj}\subset\gamma}}{\rm lk}(f(\lambda))^2+\sum_{\lambda\in\Gamma_{3,3}(F_{ij}^{(kl)})}{\rm lk}(f(\lambda))^2-1.\nonumber
\end{eqnarray}
Note that by applying Theorem \ref{mainthm} in the case of $n=6$ to the embedding $f$ restricted to $F_{ij}^{(lk)}$, we also have a similar formula as (\ref{f78}). 
Let us take the sum of both sides of (\ref{f78}) over $1\leq i<j\leq 8$ and $i,j\neq k,l$. 
%1
For a $t(\ge 4)$-cycle $\gamma$ of $K_8$ containing $\overline{kl}$, let $i$ and $j$ be the two vertices of $K_8$ which are adjacent to $\overline{kl}$ in $\gamma\ (1\leq i<j\leq 8)$. Then $\gamma$ is a $t$-cycle of $F_{ij}^{(kl)}$ or $F_{ij}^{(lk)}$. This implies that
\begin{eqnarray}\label{f781}
\sum_{\substack{1\le i<j\le 8\\ i,j\neq k,l}}
\Bigg( 
\sum_{\substack{\gamma\in\Gamma_{t}(F_{ij}^{(kl)})\\ \overline{iklj}\subset\gamma}}\!\!\!\!\! a_2(f(\gamma))+\sum_{\substack{\gamma\in\Gamma_{t}(F_{ij}^{(lk)})\\ \overline{ilkj}\subset\gamma}}\!\!\!\!\! a_2(f(\gamma))
\Bigg)=
\sum_{\substack{\gamma\in\Gamma_{t}(K_8)\\ \overline{kl}\subset\gamma}}\!\!\!\!\! a_2(f(\gamma)).
\end{eqnarray}
%2
For a $t$-cycle $\gamma$ of $K_6^{(kl)}$, let $\overline{ij}$ be an edge of $K_6^{(kl)}$ which is not contained in $\gamma$. Then $\gamma$ is a $t$-cycle of $F_{ij}^{(kl)}$ (resp. $F_{ij}^{(lk)}$) which does not contain $\overline{iklj}$ (resp. $\overline{ilkj}$). Note that there are $(15-t)$ ways to to choose such a pair of $i$ and $j$. This implies that
\begin{eqnarray}
&& \sum_{\substack{1\le i<j\le 8\\ i,j\neq k,l}}
\sum_{\substack{\gamma\in\Gamma_{t}(F_{ij}^{(kl)})\\ \overline{iklj}\not\subset\gamma}}a_2(f(\gamma))
= 
(15-t)\sum_{\gamma\in\Gamma_{t}(K_6^{(kl)})}a_2(f(\gamma)),\label{f782}\\
&& \sum_{\substack{1\le i<j\le 8\\ i,j\neq k,l}}
\sum_{\substack{\gamma\in\Gamma_{t}(F_{ij}^{(lk)})\\ \overline{ilkj}\not\subset\gamma}}a_2(f(\gamma))
= 
(15-t)\sum_{\gamma\in\Gamma_{t}(K_6^{(kl)})}a_2(f(\gamma)).\label{f783}
\end{eqnarray}
On the other hand, we have 
\begin{eqnarray}\label{f787}
&& \sum_{\substack{1\le i<j\le 8\\ i,j\neq k,l}}
\Bigg(
\sum_{\substack{\lambda=\gamma\cup\gamma'\in\Gamma_{3,5}(F_{ij}^{(kl)})\\ \gamma\in\Gamma_5(F_{ij}^{(kl)}),\ \gamma'\in\Gamma_3(F_{ij}^{(kl)})\\ \overline{iklj}\subset\gamma}}\!\!\!\!\! {\rm lk}(f(\lambda))^2
+\sum_{\substack{\lambda=\gamma\cup\gamma'\in\Gamma_{3,5}(F_{ij}^{(lk)})\\ \gamma\in\Gamma_5(F_{ij}^{(lk)}),\ \gamma'\in\Gamma_3(F_{ij}^{(lk)})\\ \overline{ilkj}\subset\gamma}}\!\!\!\!\! {\rm lk}(f(\lambda))^2
\Bigg)\\
&=&\sum_{\substack{\lambda=\gamma\cup\gamma'\in\Gamma_{3,5}(K_8)\\ \gamma\in\Gamma_5(K_8),\ \gamma'\in\Gamma_3(K_8)\\ \overline{kl}\subset\gamma}}\!\!\!\!\! {\rm lk}(f(\lambda))^2, \nonumber
\end{eqnarray}
\begin{eqnarray}
&& \sum_{\substack{1\le i<j\le 8\\ i,j\neq k,l}}
\sum_{\substack{\lambda\in\Gamma_{3,3}(F_{ij}^{(kl)})\\ \overline{iklj}\not\subset\gamma}}{\rm lk}(f(\lambda))^2
= 
9\sum_{\lambda\in\Gamma_{3,3}(K_6^{(kl)})}{\rm lk}(f(\lambda))^2,\label{f788}\\
&& \sum_{\substack{1\le i<j\le 8\\ i,j\neq k,l}}
\sum_{\substack{\lambda\in\Gamma_{3,3}(F_{ij}^{(lk)})\\ \overline{ilkj}\not\subset\gamma}}{\rm lk}(f(\lambda))^2
= 
9\sum_{\lambda\in\Gamma_{3,3}(K_6^{(kl)})}{\rm lk}(f(\lambda))^2.\label{f789}
\end{eqnarray}
By combining (\ref{f781}), (\ref{f782}), (\ref{f783}), (\ref{f787}), (\ref{f788}) and (\ref{f789}) with (\ref{f78}), we have 
\begin{eqnarray}\label{f7810}
&&2\sum_{\substack{\gamma\in\Gamma_8(K_8)\\ \overline{kl}\subset\gamma}}a_2(f(\gamma))+36\sum_{\gamma\in\Gamma_6(K_6^{(kl)})}a_2(f(\gamma))\\
&&-2\sum_{\substack{\gamma\in\Gamma_7(K_8)\\ \overline{kl}\subset\gamma}}a_2(f(\gamma))-40\sum_{\gamma\in\Gamma_5(K_6^{(kl)})}a_2(f(\gamma))\nonumber\\
&=&
\sum_{\substack{\lambda=\gamma\cup\gamma'\in\Gamma_{3,5}(K_8)\\ \gamma\in\Gamma_5(K_8),\ \gamma'\in\Gamma_3(K_8)\\ \overline{kl}\subset\gamma}}\!\!\!\!{\rm lk}(f(\lambda))^2+18\sum_{\lambda\in\Gamma_{3,3}(K_6^{(kl)})}\!\!\!\!{\rm lk}(f(\lambda))^2-30.\nonumber
\end{eqnarray}
Then, by applying Theorem \ref{mainthm} in the case of $n=6$ to the embedding $f$ restricted to $K_6^{(kl)}$, we have 
\begin{eqnarray}\label{f7811}
&& 36\sum_{\gamma\in\Gamma_6(K_6^{(kl)})}\!\!\!\!\! a_2(f(\gamma))-40\sum_{\gamma\in\Gamma_5(K_6^{(kl)})}\!\!\!\!\! a_2(f(\gamma))\\
&=&
18\Bigg(
2\bigg(
\sum_{\gamma\in\Gamma_6(K_6^{(kl)})}\!\!\!\!\! a_2(f(\gamma))-\sum_{\gamma\in\Gamma_5(K_6^{(kl)})}\!\!\!\!\! a_2(f(\gamma))
\bigg)\Bigg)-4\sum_{\gamma\in\Gamma_5(K_6^{(kl)})}\!\!\!\!\! a_2(f(\gamma))\nonumber\\
&=&
18\bigg(
\sum_{\lambda\in\Gamma_{3,3}(K_6^{(kl)})}\!\!\!\!\! a_2(f(\lambda))^2-1
\bigg)-4\sum_{\gamma\in\Gamma_5(K_6^{(kl)})}\!\!\!\!\! a_2(f(\gamma)).\nonumber
\end{eqnarray}
By combining (\ref{f7811}) with (\ref{f7810}), we have 
\begin{eqnarray}\label{fij}
&&2\sum_{\substack{\gamma\in\Gamma_8(K_8)\\ \overline{kl}\subset\gamma}}a_2(f(\gamma))
-2\sum_{\substack{\gamma\in\Gamma_7(K_8)\\ \overline{kl}\subset\gamma}}a_2(f(\gamma))
-4\sum_{\gamma\in\Gamma_5(K_6^{(kl)})}a_2(f(\gamma))\\
&=&
\sum_{\substack{\lambda=\gamma\cup\gamma'\in\Gamma_{3,5}(K_8)\\ \gamma\in\Gamma_5(K_8),\ \gamma'\in\Gamma_3(K_8)\\ \overline{kl}\subset\gamma}}{\rm lk}(f(\lambda))^2-12.\nonumber
\end{eqnarray}
Now we take the sum of both side of (\ref{fij}) over $1\leq k<l \leq8$. 
%1
For a $t$-cycle $\gamma$ of $K_8$, let $\overline{kl}$ be an edge of $K_8$ which is contained in $\gamma$. Note that there are $t$ ways to choose such two vertices $k$ and $l$. This implies that
\begin{eqnarray}
\sum_{1\leq k<l\leq8}\label{fij1}
\sum_{\substack{\gamma\in\Gamma_{t}(K_8)\\ \overline{kl}\subset\gamma}}a_2(f(\gamma))=t\sum_{\gamma\in\Gamma_{t}(K_8)}a_2(f(\gamma)).
\end{eqnarray}
%2
%3
For a $5$-cycle $\gamma$ of $K_8$, let $k$ and $l$ be two distinct vertices of $K_8$ which are not contained in $\gamma$. Then $\gamma$ is a $5$-cycle of $K_6^{(kl)}$. Note that there are three ways to choose such two vertices $k$ and $l$. This implies that
\begin{eqnarray}
\sum_{1\leq k<l\leq 8}\label{fij3}
\sum_{\gamma\in\Gamma_5(K_6^{(kl)})}a_2(f(\gamma))=3\sum_{\gamma\in\Gamma_5(K_8)}a_2(f(\gamma)).
\end{eqnarray}
%4
On the other hand, we have 
%
%For a pair of disjoint cycles $\lambda$ of $K_8$ which consisting of a 
%$5$-cycle $\gamma$ and $3$-cycle $\gamma'$, let $\overline{kl}$ be an 
%edge of $K_8$ which is contained in $\gamma$. Note that there are five 
%ways to choose such two vertices $k$ and $l$. This implies that
\begin{eqnarray}
&&\sum_{1\leq k<l\leq8}\label{fij4}
\sum_{\substack{\lambda=\gamma\cup\gamma'\in\Gamma_{3,5}(K_8)\\ \gamma\in\Gamma_5(K_8),\ \gamma'\in\Gamma_3(K_8)\\ \overline{kl}\subset\gamma}}{\rm lk}(f(\lambda))^2
=5\sum_{\lambda\in\Gamma_{3,5}(K_8)}{\rm lk}(f(\lambda))^2.
\end{eqnarray}
By combining (\ref{fij1}), (\ref{fij3}) and (\ref{fij4}) with (\ref{fij}), we have the result.
\end{proof}

As we mentioned before, (\ref{ilKK0}) has already been shown if $p=3,\ q=4$ \cite[Theorem 2.3]{MN19}. Namely for any spatial embedding $f$ of $K_{7}$, we have 
\begin{eqnarray}\label{k73433}
\sum_{\lambda\in \Gamma_{3,4}(K_{7})}{\rm lk}(f(\lambda))^{2} 
= 2 \sum_{\lambda\in \Gamma_{3,3}(K_{7})}{\rm lk}(f(\lambda))^{2}. 
\end{eqnarray}

Moreover, the following has also been shown by the authors in \cite[Lemma 2.1 (2)]{MN19}. 

\begin{Lemma}\label{kn3433}
For any spatial embedding $f$ of $K_n$ with $n\ge 7$, we have 
\begin{eqnarray*}
\sum_{\lambda\in \Gamma_{3,4}(K_{n})}{\rm lk}(f(\lambda))^{2} 
= 2(n-6) \sum_{\lambda\in \Gamma_{3,3}(K_{n})}{\rm lk}(f(\lambda))^{2}. 
\end{eqnarray*}
\end{Lemma}

On the other hand, we also have the following in the case of $n=8$. 

\begin{Lemma}\label{413b}
For any spatial embedding $f$ of $K_8$, we have 
\begin{eqnarray*}
8\sum_{\lambda\in\Gamma_{4,4}(K_8)}{\rm lk}(f(\lambda))^2+5\sum_{\lambda\in\Gamma_{3,5}(K_8)}{\rm lk}(f(\lambda))^2=36\sum_{\lambda\in\Gamma_{3,3}(K_8)}{\rm lk}(f(\lambda))^2.
\end{eqnarray*}
\end{Lemma}

\begin{proof}
For any spatial embedding $g$ of $K_{7}$, by (\ref{k73433}) we have 
\begin{eqnarray*}
\sum_{\lambda\in \Gamma_{3,4}(K_{7})}{\rm lk}(g(\lambda))^{2} 
= 2 \sum_{\lambda\in \Gamma_{3,3}(K_{7})}{\rm lk}(g(\lambda))^{2}. 
\end{eqnarray*}
Then we have the result by setting $n=8$, $a=2$, $t=3$ and $s=1$ in Lemma \ref{L1} which will be shown in $\S 3$. 
\end{proof}

\begin{proof}[Proof of Theorem \ref{lkthm}]
Let $f$ be a spatial embedding of $K_{8}$. Then by Theorem \ref{mainthm} in the case of $n=8$, we have 
\begin{eqnarray}\label{k8_0}
\sum_{\gamma\in\Gamma_8(K_8)}\!\!\!\!\! a_2(f(\gamma))-6\sum_{\gamma\in\Gamma_5(K_8)}\!\!\!\!\! a_2(f(\gamma))
= 3\sum_{\lambda\in\Gamma_{3,3}(K_8)}\!\!\!\!\! {\rm lk}(f(\lambda))^2
-63.
\end{eqnarray}
On the other hand, by Lemmas \ref{K8-48} and \ref{413}, we have 
\begin{eqnarray}
&&16\sum_{\gamma\in\Gamma_8(K_8)}a_2(f(\gamma))-96\sum_{\gamma\in\Gamma_5(K_8)}a_2(f(\gamma))\label{k8_1}\\
&=& 28\sum_{\lambda\in\Gamma_{3,3}(K_8)}{\rm lk}(f(\lambda))^2
+5\sum_{\lambda\in\Gamma_{3,5}(K_8)}{\rm lk}(f(\lambda))^2
-1008.\nonumber
\end{eqnarray}
Then by combining (\ref{k8_0}) and (\ref{k8_1}), we have (\ref{3533}), and by combining Lemma \ref{413b} and (\ref{3533}), we have (\ref{4433}). This completes the proof. 
\end{proof}

In a similar way as Lemma \ref{kn3433}, we have the following by (\ref{3533}).

\begin{Lemma}\label{LL2}
Let $n\ge 8$ be an integer. For any spatial embedding $f$ of $K_n$, we have 
\begin{eqnarray*}
\sum_{\lambda\in\Gamma_{3,5}(K_n)}{\rm lk}(f(\lambda))^2=2(n-6)(n-7)\sum_{\lambda\in\Gamma_{3,3}(K_n)}{\rm lk}(f(\lambda))^2.
\end{eqnarray*}
\end{Lemma}

\begin{proof}
Note that each pair of two disjoint $3$-cycles of $K_n$ is shared by exactly $\binom{n-6}{2}$ subgraphs isomorphic to $K_8$ if $n\ge8$. Then by applying (\ref{3533}) to the embedding $f$ restricted to each of the subgraphs of $K_n$ isomorphic to $K_8$ and taking the sum of both sides of this over all of them, we have the result.
\end{proof}

%(4.3ß)
\section{Proof of Theorem \ref{lkrefine}: General case}%4.2

In this section, we show Theorem \ref{lkrefine} in the case of $n\ge 9$. First we show two lemmas which are needed later.

\begin{Lemma}\label{L1}
Let $n\ge 7$ be an integer and $t,s$ two integers satisfying $n-1=2t+s$, where $t\ge 3$ and $s\ge 0$. 
Assume that there exist a constant $a$ such that
\begin{eqnarray}\label{31assume}
\sum_{\lambda\in\Gamma_{t,t+s}(K_{n-1})}{\rm lk}(g(\lambda))^2=a\sum_{\lambda\in\Gamma_{3,3}(K_{n-1})}{\rm lk}(g(\lambda))^2
\end{eqnarray}
for any spatial embedding $g$ of $K_{n-1}$. Then for any spatial embedding $f$ of $K_n$, we have 
\begin{eqnarray}\label{L10}
\sum_{\lambda\in\Gamma_{t,t+1}(K_n)}{\rm lk}(f(\lambda))^2=2a(n-6)\sum_{\lambda\in\Gamma_{3,3}(K_n)}{\rm lk}(f(\lambda))^2
\end{eqnarray}
if $s=0$, 
\begin{eqnarray}\label{L11}
&&2(t+1)\sum_{\lambda\in\Gamma_{t+1,t+1}(K_n)}{\rm lk}(f(\lambda))^2+(t+2)\sum_{\lambda\in\Gamma_{t,t+2}(K_n)}{\rm lk}(f(\lambda))^2\\
&=&a(n-6)(2t+3)\sum_{\lambda\in\Gamma_{3,3}(K_n)}{\rm lk}(f(\lambda))^2\nonumber
\end{eqnarray}
if $s=1$ and 
\begin{eqnarray}\label{L12}
&&(t+1)\sum_{\lambda\in\Gamma_{t+1,t+s}(K_n)}{\rm lk}(f(\lambda))^2+(t+s+1)\sum_{\lambda\in\Gamma_{t,t+s+1}(K_n)}{\rm lk}(f(\lambda))^2\\
&=&a(n-6)(2t+s+2)\sum_{\lambda\in\Gamma_{3,3}(K_n)}{\rm lk}(f(\lambda))^2\nonumber
\end{eqnarray}
if $s\ge 2$. 
%However, $l\ge2$ if and only if $n\ge9$.
\end{Lemma}

\begin{proof}
Let $f$ be a spatial embedding of $K_n$. First we assume that $s\ge 1$. For the embedding $f$ restricted to $F_{ij}^{(m)}$, by (\ref{31assume}) we have 
\begin{eqnarray}\label{l11fijm}
&& \sum_{\substack{\lambda=\gamma\cup\gamma'\in\Gamma_{t+1,t+s}(F_{ij}^{(m)})\\ \gamma\in\Gamma_{t+1}(F_{ij}^{(m)}),\ \gamma'\in\Gamma_{t+s}(F_{ij}^{(m)})\\ \overline{imj}\subset\gamma}}{\rm lk}(f(\lambda))^2
+\sum_{\substack{\lambda=\gamma\cup\gamma'\in\Gamma_{t,t+s+1}(F_{ij}^{(m)})\\ \gamma\in\Gamma_{t}(F_{ij}^{(m)}),\ \gamma'\in\Gamma_{t+s+1}(F_{ij}^{(m)})\\ \overline{imj}\subset\gamma'}}{\rm lk}(f(\lambda))^2 \\
&& +\sum_{\substack{\lambda\in\Gamma_{t,t+s}(K_{n-1}^{(m)})\\ \overline{ij}\not\subset\lambda}}{\rm lk}(f(\lambda))^2 \nonumber\\
&=& a\Bigg( \sum_{\substack{\lambda=\gamma\cup\gamma'\in\Gamma_{3,4}(F_{ij}^{(m)})\\ \gamma\in\Gamma_4(F_{ij}^{(m)}),\ \gamma'\in\Gamma_3(F_{ij}^{(m)})\\ \overline{imj}\subset\gamma}}{\rm lk}(f(\lambda))^2
+\sum_{\substack{\lambda\in\Gamma_{3,3}(K_{n-1}^{(m)})\\ \overline{ij}\not\subset\lambda}}{\rm lk}(f(\lambda))^2\Bigg).\nonumber
\end{eqnarray}
Let us take the sum of both side of (\ref{l11fijm}) over $1\leq i<j\leq n$ and $i, j\neq m$. 
For a pair of disjoint cycles $\lambda$ of $K_n$ 
consisting of a $(t+1)$-cycle $\gamma$ which contains 
the vertex $m$ and a $(t+s)$-cycle $\gamma'$, let $i$ and $j$ 
be the two vertices of $K_n$ which are adjacent to $m$ in $\gamma$. 
Then $\lambda$ is a pair of disjoint cycles of $F_{ij}^{(m)}$ 
consisting of a $(t+1)$-cycle $\gamma$ which contains $\overline{imj}$ 
and a $(t+s)$-cycle $\gamma'$. This implies that 
\begin{eqnarray}\label{l11t1tl}
&& \sum_{\substack{1\leq i<j\leq n\\ i,j\neq m}}\sum_{\substack{\lambda=\gamma\cup\gamma'\in\Gamma_{t+1,t+s}(F_{ij}^{(m)})\\ \gamma\in\Gamma_{t+1}(F_{ij}^{(m)}),\ \gamma'\in\Gamma_{t+s}(F_{ij}^{(m)}) \\ \overline{imj}\subset\gamma}}\!\!\!\!\! {\rm lk}(f(\lambda))^2
= 
\sum_{\substack{\lambda=\gamma\cup\gamma'\in\Gamma_{t+1,t+s}(K_n)\\ \gamma\in\Gamma_{t+1}(K_n),\ \gamma'\in\Gamma_{t+s}(K_n) \\ m\subset\gamma}}\!\!\!\!\! {\rm lk}(f(\lambda))^2. 
\end{eqnarray}
In the same way as (\ref{l11t1tl}), we also have 
%2
%For a pair of disjoint cycles $\lambda$ of $K_n$ consisting 
%of a $t$-cycle $\gamma$ which contains the vertex $m$ and a 
%$(t+s+1)$-cycle $\gamma'$, let $i$ and $j$ be the two vertices 
%of $K_n$ which are adjacent to $m$ in $\gamma'$. 
%Then $\lambda$ is a pair of disjoint cycles of $F_{ij}^{(m)}$ 
%consisting of a $t$-cycle $\gamma$ and a $(t+s+1)$-cycle $\gamma'$ 
%which contains $\overline{imj}$. This implies that 
\begin{eqnarray}
&& \sum_{\substack{1\leq i<j\leq n\\ i,j\neq m}}\sum_{\substack{\lambda=\gamma\cup\gamma'\in\Gamma_{t,t+s+1}(F_{ij}^{(m)})\\ \gamma\in\Gamma_t(F_{ij}^{(m)}),\ \gamma'\in\Gamma_{t+s+1}(F_{ij}^{(m)}) \\ \overline{imj}\subset\gamma'}}\!\!\!\!\! {\rm lk}(f(\lambda))^2
= 
\sum_{\substack{\lambda=\gamma\cup\gamma'\in\Gamma_{t,t+s+1}(K_n)\\ \gamma\in\Gamma_t(K_n),\ \gamma'\in\Gamma_{t+s+1}(K_n) \\ m\subset\gamma'}}\!\!\!\!\! {\rm lk}(f(\lambda))^2,\label{l11ttl1}\\
&& \sum_{\substack{1\leq i<j\leq n\\ i,j\neq m}}\sum_{\substack{\lambda=\gamma\cup\gamma'\in\Gamma_{3,4}(F_{ij}^{(m)})\\ \gamma\in\Gamma_{4}(F_{ij}^{(m)}),\ \gamma'\in\Gamma_3(F_{ij}^{(m)}) \\ \overline{imj}\subset\gamma}}\!\!\!\!\! {\rm lk}(f(\lambda))^2
= \sum_{\substack{\lambda=\gamma\cup\gamma'\in\Gamma_{3,4}(K_n)\\ \gamma\in\Gamma_4(K_n),\ \gamma'\in\Gamma_3(K_n)\\ 
m\subset\gamma}}\!\!\!\!\! {\rm lk}(f(\lambda))^2. \label{l1143}
\end{eqnarray}
%
%3
For a pair of disjoint cycles $\lambda$ of $K_{n-1}^{(m)}$ 
consisting of a $t$-cycle and a $(t+s)$-cycle, let $\overline{ij}$ 
be an edge of $K_{n-1}^{(m)}$ which is not contained in $\lambda$. 
Note that there are $\binom{n-1}{2}-(2t+s)$ ways to choose such a 
pair of $i$ and $j$. This implies that 
\begin{eqnarray}\label{l11ttl}
\sum_{\substack{1\leq i<j\leq n\\ i,j\neq m}}\sum_{\substack{\lambda\in\Gamma_{t,t+s}(K_{n-1}^{(m)})\\ \overline{ij}\not\subset\lambda}}\!\!\!\!\!\!\!\!\!\! {\rm lk}(f(\lambda))^2
= \left( \binom{n-1}{2}-(2t+s) \right)\!\!\! \sum_{\lambda\in\Gamma_{t,t+s}(K_{n-1}^{(m)})}\!\!\!\!\!\!\!\!\!\! {\rm lk}(f(\lambda))^2. 
\end{eqnarray}
%4
%For a pair of disjoint cycles $\lambda$ of $K_n$ consisting of 
%a $4$-cycle $\gamma$ which contains the vertex $m$ and a $3$-cycle 
%$\gamma'$, let $i$ and $j$ be the two vertices of $K_n$ which are 
%adjacent to $m$ in $\gamma$. Then $\lambda$ is a pair of disjoint 
%cycles of $F_{ij}^{(m)}$ consisting of a $4$-cycle $\gamma$ which 
%contains $\overline{imj}$ and a $3$-cycle $\gamma'$. This implies that 
%
%5
%For a pair of disjoint $3$-cycles $\lambda$ of $K_{n-1}^{(m)}$, 
%let $\overline{ij}$ be an edge of $K_{n-1}^{(m)}$ which is not 
%contained in $\lambda$. Note that there are $\binom{n-1}{2}-6$ 
%ways to choose such a pair of $i$ and $j$. This implies that 
%
In the same way as (\ref{l11ttl}), we also have 
\begin{eqnarray}\label{l1133}
\sum_{\substack{1\leq i<j\leq n\\ i,j\neq m}}\sum_{\substack{\lambda\in\Gamma_{3,3}(K_{n-1}^{(m)})\\ \overline{ij}\not\subset\lambda}}\!\!\!\!\! {\rm lk}(f(\lambda))^2
= \left( \binom{n-1}{2}-6 \right)\sum_{\lambda\in\Gamma_{3,3}(K_{n-1}^{(m)})}\!\!\!\!\! {\rm lk}(f(\lambda))^2.
\end{eqnarray}
By combining (\ref{l11t1tl}), (\ref{l11ttl1}),  (\ref{l1143}), (\ref{l11ttl}) and (\ref{l1133}) with (\ref{l11fijm}), we have 
\begin{eqnarray}\label{l11kn-1}
&&\sum_{\substack{\lambda=\gamma\cup\gamma'\in\Gamma_{t+1,t+s}(K_n)\\ \gamma\in\Gamma_{t+1}(K_n),\ \gamma'\in\Gamma_{t+s}(K_n)\\ m\subset\gamma}}{\rm lk}(f(\lambda))^2
+\sum_{\substack{\lambda=\gamma\cup\gamma'\in\Gamma_{t,t+s+1}(K_n)\\ \gamma\in\Gamma_{t}(K_n),\ \gamma'\in\Gamma_{t+s+1}(K_n)\\ m\subset\gamma'}}{\rm lk}(f(\lambda))^2\\
&&+\left( \binom{n-1}{2}-(2t+s) \right)\sum_{\lambda\in\Gamma_{t,t+s}(K_{n-1}^{(m)})}{\rm lk}(f(\lambda))^2\nonumber\\ 
&=&
a\sum_{\substack{\lambda=\gamma\cup\gamma'\in\Gamma_{3,4}(K_n)\\ \gamma\in\Gamma_4(K_n),\ \gamma'\in\Gamma_3(K_n)\\ m\subset\gamma}}{\rm lk}(f(\lambda))^2+a\left( \binom{n-1}{2}-6 \right)\sum_{\lambda\in\Gamma_{3,3}(K_{n-1}^{(m)})}{\rm lk}(f(\lambda))^2.\nonumber
\end{eqnarray}
Then for the embedding $f$ restricted to $K_{n-1}^{(m)}$, by the assumption we have 
\begin{eqnarray}\label{l11kn-1m}
\sum_{\lambda\in\Gamma_{t,t+s}(K_{n-1}^{(m)})}{\rm lk}(f(\lambda))^2=a\sum_{\lambda\in\Gamma_{3,3}(K_{n-1}^{(m)})}{\rm lk}(f(\lambda))^2.
\end{eqnarray} 
By combining (\ref{l11kn-1m}) with (\ref{l11kn-1}), we have 
\begin{eqnarray}\label{l11kn-1s}
&&\sum_{\substack{\lambda=\gamma\cup\gamma'\in\Gamma_{t+1,t+s}(K_n)\\ \gamma\in\Gamma_{t+1}(K_n),\ \gamma'\in\Gamma_{t+s}(K_n)\\ m\subset\gamma}}{\rm lk}(f(\lambda))^2
+\sum_{\substack{\lambda=\gamma\cup\gamma'\in\Gamma_{t,t+s+1}(K_n)\\ \gamma\in\Gamma_{t}(K_n),\ \gamma'\in\Gamma_{t+s+1}(K_n)\\ m\subset\gamma'}}{\rm lk}(f(\lambda))^2\\
&=&a\sum_{\substack{\lambda=\gamma\cup\gamma'\in\Gamma_{3,4}(K_n)\\ \gamma\in\Gamma_4(K_n),\ \gamma'\in\Gamma_3(K_n)\\ m\subset\gamma}}{\rm lk}(f(\lambda))^2+a(2t+s-6)\sum_{\lambda\in\Gamma_{3,3}(K_{n-1}^{(m)})}{\rm lk}(f(\lambda))^2.\nonumber
\end{eqnarray}

Now we assume that $s\ge 2$. Let us take the sum of both sides of (\ref{l11kn-1s}) over $m=1,2,\ldots ,n$. 
For a pair of disjoint cycles $\lambda$ of $K_n$ consisting of 
a $(t+1)$-cycle $\gamma$ and a $(t+s)$-cycle $\gamma'$, let $m$ 
be a vertex of $K_n$ which is contained in $\gamma$. Note that 
there are $t+1$ ways to choose such a vertex $m$. This implies that  
\begin{eqnarray}\label{l12t1tl2}
\sum_{m=1}^{n}\sum_{\substack{\lambda=\gamma\cup\gamma'\in\Gamma_{t+1,t+s}(K_n)\\ \gamma\in\Gamma_{t+1}(K_n),\ \gamma'\in\Gamma_{t+s}(K_n)\\ m\subset\gamma}}\!\!\!\!\! \!\!\!\!\! {\rm lk}(f(\lambda))^2
= (t+1)\sum_{\lambda\in\Gamma_{t+1,t+s}(K_n)}\!\!\!\!\! {\rm lk}(f(\lambda))^2. 
\end{eqnarray}
In the same way as (\ref{l12t1tl2}), we also have 
\begin{eqnarray}
&& \sum_{m=1}^{n}\sum_{\substack{\lambda=\gamma\cup\gamma'\in\Gamma_{t,t+s+1}(K_n)\\ \gamma\in\Gamma_t(K_n),\ \gamma'\in\Gamma_{t+s+1}(K_n)\\ m\subset\gamma'}}\!\!\!\!\! {\rm lk}(f(\lambda))^2
= (t+s+1)\sum_{\lambda\in\Gamma_{t,t+s+1}(K_n)}\!\!\!\!\! {\rm lk}(f(\lambda))^2, \label{l12ttl12}\\
&& \sum_{m=1}^{n}\sum_{\substack{\lambda=\gamma\cup\gamma'\in\Gamma_{3,4}(K_n)\\ \gamma\in\Gamma_4(K_n),\ \gamma'\in\Gamma_3(K_n)\\ m\subset\gamma}}\!\!\!\!\! {\rm lk}(f(\lambda))^2
= 4\sum_{\lambda\in\Gamma_{3,4}(K_n)}\!\!\!\!\! {\rm lk}(f(\lambda))^2. \label{l12432}
\end{eqnarray}
%4
For a pair of two disjoint $3$-cycles $\lambda$ of $K_n$, 
let $m$ be a vertex of $K_n$ which is not contained in $\lambda$. 
Then $\lambda$ is a pair of two disjoint $3$-cycles of $K_{n-1}^{(m)}$. 
Note that there are $n-6$ ways to choose such a vertex $m$. 
This implies that 
\begin{eqnarray}\label{l12332}
\sum_{m=1}^{n}\sum_{\lambda\in\Gamma_{3,3}(K_{n-1}^{(m)})}{\rm lk}(f(\lambda))^2
= (n-6)\sum_{\lambda\in\Gamma_{3,3}(K_n)}{\rm lk}(f(\lambda))^2.
\end{eqnarray}
By combining (\ref{l12t1tl2}), (\ref{l12ttl12}), (\ref{l12432}) and (\ref{l12332}) with (\ref{l11kn-1s}), we have
\begin{eqnarray}\label{l12kn}
&&(t+1)\sum_{\lambda\in\Gamma_{t+1,t+s}(K_n)}{\rm lk}(f(\lambda))^2+(t+s+1)\sum_{\lambda\in\Gamma_{t,t+s+1}(K_n)}{\rm lk}(f(\lambda))^2\\
&=&4a\sum_{\lambda\in\Gamma_{3,4}(K_n)}{\rm lk}(f(\lambda))^2+a(2t+s-6)(n-6)\sum_{\lambda\in\Gamma_{3,3}(K_n)}{\rm lk}(f(\lambda))^2.\nonumber
\end{eqnarray}
Then by (\ref{l12kn}) and Lemma \ref{kn3433}, we have (\ref{L12}). 

%%%

Next, assume that $s=1$. Note that we also have 
\begin{eqnarray}\label{l11t1t12}
\sum_{m=1}^{n}\sum_{\substack{\lambda\in\Gamma_{t+1,t+1}(K_n)\\ m\subset\lambda}}{\rm lk}(f(\lambda))^2
=2(t+1)\sum_{\lambda\in\Gamma_{t+1,t+1}(K_n)}{\rm lk}(f(\lambda))^2 
\end{eqnarray}
in a similar way as (\ref{l12t1tl2}). 
By combining (\ref{l12ttl12}), (\ref{l12432}), (\ref{l12332}) and (\ref{l11t1t12}) with (\ref{l11kn-1s}), we have
\begin{eqnarray}\label{l11kn}
&&2(t+1)\sum_{\lambda\in\Gamma_{t+1,t+1}(K_n)}{\rm lk}(f(\lambda))^2+(t+2)\sum_{\lambda\in\Gamma_{t,t+2}(K_n)}{\rm lk}(f(\lambda))^2\\
&=&4a\sum_{\lambda\in\Gamma_{3,4}(K_n)}{\rm lk}(f(\lambda))^2+a(2t-5)(n-6)\sum_{\lambda\in\Gamma_{3,3}(K_n)}{\rm lk}(f(\lambda))^2.\nonumber
\end{eqnarray}
Then by (\ref{l11kn}) and Lemma \ref{kn3433}, we have (\ref{L11}).

%%%
Finally, assume that $s=0$. By (\ref{31assume}) we have 
\begin{eqnarray}\label{l10fijm}
&&\sum_{\substack{\lambda\in\Gamma_{t+1,t}(F_{ij}^{(m)})\\ \gamma\in\Gamma_{t+1}(F_{ij}^{(m)}),\ \gamma'\in\Gamma_t(F_{ij}^{(m)})\\ \overline{imj}\subset\gamma}}{\rm lk}(f(\lambda))^2+\sum_{\substack{\lambda\in\Gamma_{t,t}(K_{n-1}^{(m)})\\ \overline{ij}\not\subset\lambda}}{\rm lk}(f(\lambda))^2\\
&=&a\Bigg( \sum_{\substack{\lambda\in\Gamma_{3,4}(F_{ij}^{(m)})\\ \gamma\in\Gamma_4(F_{ij}^{(m)}),\ \gamma'\in\Gamma_3(F_{ij}^{(m)})\\ \overline{imj}\subset\gamma}}{\rm lk}(f(\lambda))^2+\sum_{\substack{\lambda\in\Gamma_{3,3}(K_{n-1}^{(m)})\\ \overline{ij}\not\subset\lambda}}{\rm lk}(f(\lambda))^2\Bigg).\nonumber
\end{eqnarray}
Let us take the sum of both side of (\ref{l10fijm}) over $1\leq i<j\leq n$ and $i, j\neq m$. Note that (\ref{l11t1tl}) and (\ref{l11ttl}) also hold if $s=0$. By combining (\ref{l11t1tl}), (\ref{l1143}), (\ref{l11ttl}) and (\ref{l1133}) with (\ref{l10fijm}), we have
\begin{eqnarray}\label{l10kn-1}
&&\sum_{\substack{\lambda=\gamma\cup\gamma'\in\Gamma_{t+1,t}(K_n)\\ \gamma\in\Gamma_{t+1}(K_n),\ \gamma'\in\Gamma_t(K_n)\\ m\subset\gamma}}{\rm lk}(f(\lambda))^2+\left( \binom{n-1}{2}-2t \right)\sum_{\lambda\in\Gamma_{t,t}(K_{n-1}^{(m)})}{\rm lk}(f(\lambda))^2\\ 
&=&
a\sum_{\substack{\lambda=\gamma\cup\gamma'\in\Gamma_{3,4}(K_n)\\ \gamma\in\Gamma_4(K_n),\ \gamma'\in\Gamma_3(K_n)\\ m\subset\gamma}}{\rm lk}(f(\lambda))^2+a\left( \binom{n-1}{2}-6 \right)\sum_{\lambda\in\Gamma_{3,3}(K_{n-1}^{(m)})}{\rm lk}(f(\lambda))^2.\nonumber
\end{eqnarray}
Then for the embedding $f$ restricted to $K_{n-1}^{(m)}$, by the assumption we have
\begin{eqnarray}\label{l10kn-1m}
\sum_{\lambda\in\Gamma_{t,t}(K_{n-1}^{(m)})}{\rm lk}(f(\lambda))^2=a\sum_{\lambda\in\Gamma_{3,3}(K_{n-1}^{(m)})}{\rm lk}(f(\lambda))^2.
\end{eqnarray}
By combining (\ref{l10kn-1}) and (\ref{l10kn-1m}), we have
\begin{eqnarray}\label{l10kn-1s}
&&\sum_{\substack{\lambda=\gamma\cup\gamma'\in\Gamma_{t+1,t}(K_n)\\ \gamma\in\Gamma_{t+1}(K_n),\ \gamma'\in\Gamma_t(K_n)\\ m\subset\gamma}}{\rm lk}(f(\lambda))^2\\
&=&a\sum_{\substack{\lambda=\gamma\cup\gamma'\in\Gamma_{3,4}(K_n)\\ \gamma\in\Gamma_4(K_n),\ \gamma'\in\Gamma_3(K_n)\\ m\subset\gamma}}{\rm lk}(f(\lambda))^2+2a(t-3)\sum_{\lambda\in\Gamma_{3,3}(K_{n-1}^{(m)})}{\rm lk}(f(\lambda))^2.\nonumber
\end{eqnarray}
Now we take the sum of both sides of (\ref{l10kn-1s}) over $m=1,2,\cdots ,n$. 
Note that (\ref{l12t1tl2}) also holds if $s=0$. By combining (\ref{l12t1tl2}), (\ref{l12432}) and (\ref{l12332}) with (\ref{l10kn-1s}), we have 
\begin{eqnarray}\label{l10kn}
&&(t+1)\sum_{\lambda\in\Gamma_{t+1,t}(K_n)}{\rm lk}(f(\lambda))^2\\
&=&4a\sum_{\lambda\in\Gamma_{3,4}(K_n)}{\rm lk}(f(\lambda))^2+2a(t-3)(n-6)\sum_{\lambda\in\Gamma_{3,3}(K_n)}{\rm lk}(f(\lambda))^2.\nonumber
\end{eqnarray}
Then by (\ref{l10kn}) and Lemma \ref{kn3433}, we have (\ref{L10}). This completes the proof. 
\end{proof}

\begin{Lemma}\label{L2}
Let $n\ge 8$ be an integer and $t,s$ two integers satisfying $n-2=2t+s$, where $t\ge 3$ and $s\ge 0$. 
Assume that there exist a constant $a$ such that 
\begin{eqnarray}\label{32assume}
\sum_{\lambda\in\Gamma_{t,t+s}(K_{n-2})}{\rm lk}(g(\lambda))^2=a\sum_{\lambda\in\Gamma_{3,3}(K_{n-2})}{\rm lk}(g(\lambda))^2 
\end{eqnarray}
for any spatial embedding $g$ of $K_{n-2}$. Then for any spatial embedding $f$ of $K_n$, we have 
\begin{eqnarray}\label{L20}
\sum_{\lambda\in\Gamma_{t,t+2}(K_n)}{\rm lk}(f(\lambda))^2
=2a(n-6)(n-7)\sum_{\lambda\in\Gamma_{3,3}(K_n)}{\rm lk}(f(\lambda))^2
\end{eqnarray}
if $s=0$, 
\begin{eqnarray}\label{L22}
&&2(t+2)\sum_{\lambda\in\Gamma_{t+2,t+2}(K_n)}{\rm lk}(f(\lambda))^2
+(t+4)\sum_{\lambda\in\Gamma_{t,t+4}(K_n)}{\rm lk}(f(\lambda))^2\\
&=&2a(n-6)(n-7)(t+3)\sum_{\lambda\in\Gamma_{3,3}(K_n)}{\rm lk}(f(\lambda))^2\nonumber
\end{eqnarray}
if $s=2$ and 
\begin{eqnarray}\label{L23}
&&(t+2)\sum_{\lambda\in\Gamma_{t+2,t+s}(K_n)}{\rm lk}(f(\lambda))^2+(t+s+2)\sum_{\lambda\in\Gamma_{t,t+s+2}(K_n)}{\rm lk}(f(\lambda))^2\\
&=&a(n-6)(n-7)(2t+s+4)\sum_{\lambda\in\Gamma_{3,3}(K_n)}{\rm lk}(f(\lambda))^2\nonumber
\end{eqnarray}
if $s=1$ or $s\ge 3$. 
\end{Lemma}

\begin{proof}
Let $f$ be a spatial embedding of $K_n$. First we assume that $s\ge 1$. Then for the embedding $f$ restricted to $F_{ij}^{(kl)}$ $(1\le i<j\le n,\ i,j\neq k,l)$, by (\ref{32assume}) we have 
\begin{eqnarray}\label{l21fijpq}
&&\sum_{\substack{\lambda=\gamma\cup\gamma'\in\Gamma_{t+2,t+s}(F_{ij}^{(kl)})\\ \gamma\in\Gamma_{t+2}(F_{ij}^{(kl)}),\ \gamma'\in\Gamma_{t+s}(F_{ij}^{(kl)}) \\ \overline{iklj}\subset\gamma}}{\rm lk}(f(\lambda))^2
+\sum_{\substack{\lambda=\gamma\cup\gamma'\in\Gamma_{t,t+s+2}(F_{ij}^{(kl)})\\ \gamma\in\Gamma_{t}(F_{ij}^{(kl)}),\ \gamma'\in\Gamma_{t+s+2}(F_{ij}^{(kl)}) \\ \overline{iklj}\subset\gamma'}}{\rm lk}(f(\lambda))^2\\
&&+\sum_{\substack{\lambda\in\Gamma_{t,t+s}(F_{ij}^{(kl)})\\ \overline{iklj}\not\subset\lambda}}{\rm lk}(f(\lambda))^2\nonumber\\
&=&a\Bigg( \sum_{\substack{\lambda=\gamma\cup\gamma'\in\Gamma_{3,5}(F_{ij}^{(kl)})\\ \gamma\in\Gamma_5(F_{ij}^{(kl)}),\ \gamma'\in\Gamma_3(F_{ij}^{(kl)})\\ \overline{iklj}\subset\gamma}}{\rm lk}(f(\lambda))^2
+\sum_{\substack{\lambda\in\Gamma_{3,3}(F_{ij}^{(kl)})\\ \overline{iklj}\not\subset\lambda}}{\rm lk}(f(\lambda))^2\Bigg).\nonumber
\end{eqnarray}
Note that for the embedding $f$ restricted to $F_{ij}^{(lk)}$, we also have a similar formula as (\ref{l21fijpq}). 
Let us take the sum of both sides of (\ref{l21fijpq}) over $1\leq i<j\leq n$ and $i, j\neq k, l$. 
%1
For a pair of disjoint cycles $\lambda$ of $K_n$ consisting of a 
$(t+2)$-cycle $\gamma$ which contains the edge $\overline{kl}$ and a 
$(t+s)$-cycle $\gamma'$, let $i$ and $j$ be two distinct vertices of 
$K_n$ which are adjacent to $\overline{kl}$ in $\gamma$. Then $\lambda$ 
is a pair of disjoint cycles of $F_{ij}^{(kl)}$ or $F_{ij}^{(lk)}$ 
consisting of a $(t+2)$-cycle $\gamma$ which contains $\overline{iklj}$ 
or $\overline{ilkj}$ and a $(t+s)$-cycle $\gamma'$. This implies that 
\begin{eqnarray}\label{l21t2tl}
&&\sum_{\substack{1\leq i<j\leq n\\ i,j\neq k,l}}\Bigg(
\sum_{\substack{\lambda=\gamma\cup\gamma'\in\Gamma_{t+2,t+s}(F_{ij}^{(kl)})\\ \gamma\in\Gamma_{t+2}(F_{ij}^{(kl)}),\ \gamma'\in\Gamma_{t+s}(F_{ij}^{(kl)})\\ \overline{iklj}\subset\gamma}}\!\!\!\!\!\!\!\!\!\!\!\!\!\!\! {\rm lk}(f(\lambda))^2
+\sum_{\substack{\lambda=\gamma\cup\gamma'\in\Gamma_{t+2,t+s}(F_{ij}^{(lk)})\\ \gamma\in\Gamma_{t+2}(F_{ij}^{(lk)}),\ \gamma'\in\Gamma_{t+s}(F_{ij}^{(lk)})\\ \overline{ilkj}\subset\gamma}}\!\!\!\!\!\!\!\!\!\!\!\!\!\!\! {\rm lk}(f(\lambda))^2
\Bigg)\\
&=&
\sum_{\substack{\lambda=\gamma\cup\gamma'\in\Gamma_{t+2,t+s}(K_n)\\ \gamma\in\Gamma_{t+2}(K_n),\ \gamma'\in\Gamma_{t+s}(K_n)\\ \overline{kl}\subset\gamma}}\!\!\!\!\!\!\!\!\!\! {\rm lk}(f(\lambda))^2.\nonumber
\end{eqnarray}
%2
In the same way as (\ref{l21t2tl}), we also have 
%
%For a pair of disjoint cycles $\lambda$ of $K_n$ consisting of a $t$-cycle 
%$\gamma$ and a $(t+s+2)$-cycle $\gamma'$ which contains the edge 
%$\overline{kl}$, let $i$ and $j$ be two distinct vertices of $K_n$ 
%which are adjacent to $\overline{kl}$ in $\gamma'$. Then $\lambda$ is a 
%pair of disjoint cycles of $F_{ij}^{(kl)}$ or $F_{ij}^{(lk)}$ consisting 
%of a $t$-cycle $\gamma$ and a $(t+s+2)$-cycle $\gamma'$ which contains 
%$\overline{iklj}$ or $\overline{ilkj}$. This implies that 
\begin{eqnarray}\label{l21ttl2}
&&\sum_{\substack{1\leq i<j\leq n\\ i,j\neq k,l}}\Bigg(
\sum_{\substack{\lambda=\gamma\cup\gamma'\in\Gamma_{t,t+s+2}(F_{ij}^{(kl)})\\ \gamma\in\Gamma_t(F_{ij}^{(kl)}),\ \gamma'\in\Gamma_{t+s+2}(F_{ij}^{(kl)})\\ \overline{iklj}\subset\gamma'}}\!\!\!\!\!\!\!\!\!\!\!\!\!\!\! {\rm lk}(f(\lambda))^2
+\sum_{\substack{\lambda=\gamma\cup\gamma'\in\Gamma_{t,t+s+2}(F_{ij}^{(lk)})\\ \gamma\in\Gamma_t(F_{ij}^{(lk)}),\ \gamma'\in\Gamma_{t+s+2}(F_{ij}^{(lk)})\\ \overline{ilkj}\subset\gamma'}}\!\!\!\!\!\!\!\!\!\!\!\!\!\!\! {\rm lk}(f(\lambda))^2
\Bigg)\\
&=&
\sum_{\substack{\lambda=\gamma\cup\gamma'\in\Gamma_{t,t+s+2}(K_n)\\ \gamma\in\Gamma_t(K_n),\ \gamma'\in\Gamma_{t+s+2}(K_n)\\ \overline{kl}\subset\gamma'}}\!\!\!\!\!\!\!\!\!\! {\rm lk}(f(\lambda))^2,\nonumber
\end{eqnarray}
\begin{eqnarray}\label{l2153}
&&\sum_{\substack{1\leq i<j\leq n\\ i,j\neq k,l}}\Bigg(
\sum_{\substack{\lambda=\gamma\cup\gamma'\in\Gamma_{3,5}(F_{ij}^{(kl)})\\ \gamma\in\Gamma_5(F_{ij}^{(kl)}),\ \gamma'\in\Gamma_3(F_{ij}^{(kl)})\\ \overline{iklj}\subset\gamma}}\!\!\!\!\!\!\!\!\!\! {\rm lk}(f(\lambda))^2
+\sum_{\substack{\lambda=\gamma\cup\gamma'\in\Gamma_{3,5}(F_{ij}^{(lk)})\\ \gamma\in\Gamma_5(F_{ij}^{(lk)}),\ \gamma'\in\Gamma_3(F_{ij}^{(lk)})\\ \overline{ilkj}\subset\gamma}}\!\!\!\!\!\!\!\!\!\! {\rm lk}(f(\lambda))^2
\Bigg)\\
&=&
\sum_{\substack{\lambda=\gamma\cup\gamma'\in\Gamma_{3,5}(K_n)\\ \gamma\in\Gamma_5(K_n),\ \gamma'\in\Gamma_3(K_n)\\ \overline{kl}\subset\gamma}}\!\!\!\!\!\!\!\!\!\! {\rm lk}(f(\lambda))^2.\nonumber
\end{eqnarray}
%
%3
For a pair of disjoint cycles $\lambda$ of $K_n$ consisting of a $t$-cycle and a $(t+s)$-cycle, let $\overline{ij}$ be an edge of $K_{n-2}^{(kl)}$ which is not contained in $\lambda$. Then $\lambda$ is a pair of disjoint cycles of $F_{ij}^{(kl)}$ (resp. $F_{ij}^{(lk)}$) which does not contain $\overline{iklj}$ (resp. $\overline{ilkj}$). Note that there are $\binom{n-2}{2}-(2t+s)$ ways to choose such a pair of $i$ and $j$. This implies that 
\begin{eqnarray}
&& \sum_{\substack{1\leq i<j\leq n\\ i,j\neq k,l}}
\sum_{\substack{\lambda\in\Gamma_{t,t+s}(F_{ij}^{(kl)})\\ \overline{iklj}\not\subset\lambda}}\!\!\!\!\! {\rm lk}(f(\lambda))^2
= \left( \binom{n-2}{2}-(2t+s) \right)\sum_{\lambda\in\Gamma_{t,t+s}(K_{n-2}^{(kl)})}\!\!\!\!\! {\rm lk}(f(\lambda))^2,\label{l21ttlpq}\\
&& \sum_{\substack{1\leq i<j\leq n\\ i,j\neq k,l}}
\sum_{\substack{\lambda\in\Gamma_{t,t+s}(F_{ij}^{(lk)})\\ \overline{ilkj}\not\subset\lambda}}\!\!\!\!\! {\rm lk}(f(\lambda))^2
= \left( \binom{n-2}{2}-(2t+s) \right)\sum_{\lambda\in\Gamma_{t,t+s}(K_{n-2}^{(kl)})}\!\!\!\!\! {\rm lk}(f(\lambda))^2.\label{l21ttlqp}
\end{eqnarray}

In the same way as (\ref{l21ttlpq}) and (\ref{l21ttlqp}), we also have 
\begin{eqnarray}
&& \sum_{\substack{1\leq i<j\leq n\\ i,j\neq k,l}}
\sum_{\substack{\lambda\in\Gamma_{3,3}(F_{ij}^{(kl)})\\ \overline{iklj}\not\subset\lambda}}{\rm lk}(f(\lambda))^2
= \left( \binom{n-2}{2}-6 \right)\sum_{\lambda\in\Gamma_{3,3}(K_{n-2}^{(kl)})}{\rm lk}(f(\lambda))^2,\label{l2133pq}\\
&& \sum_{\substack{1\leq i<j\leq n\\ i,j\neq k,l}}
\sum_{\substack{\lambda\in\Gamma_{3,3}(F_{ij}^{(lk)})\\ \overline{ilkj}\not\subset\lambda}}{\rm lk}(f(\lambda))^2
= \left( \binom{n-2}{2}-6 \right)\sum_{\lambda\in\Gamma_{3,3}(K_{n-2}^{(kl)})}{\rm lk}(f(\lambda))^2.\label{l2133qp}
\end{eqnarray}
By combining (\ref{l21t2tl}), (\ref{l21ttl2}), (\ref{l2153}), (\ref{l21ttlpq}), (\ref{l21ttlqp}), (\ref{l2133pq}) and (\ref{l2133qp}) with (\ref{l21fijpq}), we have 
\begin{eqnarray}\label{l21kn-2}
&&\sum_{\substack{\lambda=\gamma\cup\gamma'\in\Gamma_{t+2,t+s}(K_n)\\ \gamma\in\Gamma_{t+2}(K_n),\ \gamma'\in\Gamma_{t+s}(K_n)\\ \overline{kl}\subset\gamma}}{\rm lk}(f(\lambda))^2
+\sum_{\substack{\lambda=\gamma\cup\gamma'\in\Gamma_{t,t+s+2}(K_n)\\ \gamma\in\Gamma_t(K_n),\ \gamma'\in\Gamma_{t+s+2}(K_n)\\ \overline{kl}\subset\gamma'}}{\rm lk}(f(\lambda))^2\\
&&+2\left( \binom{n-2}{2}-(2t+s) \right)\sum_{\lambda\in\Gamma_{t,t+s}(K_{n-2}^{(kl)})}{\rm lk}(f(\lambda))^2\nonumber\\
&=&
a\Bigg( \sum_{\substack{\lambda=\gamma\cup\gamma'\in\Gamma_{3,5}(K_n)\\ \gamma\in\Gamma_5(K_n),\ \gamma'\in\Gamma_3(K_n)\\ \overline{kl}\subset\gamma}}{\rm lk}(f(\lambda))^2
+2\left( \binom{n-2}{2}-6 \right)\sum_{\lambda\in\Gamma_{3,3}(K_{n-2}^{(kl)})}{\rm lk}(f(\lambda))^2\Bigg).\nonumber
\end{eqnarray}
Then for the embedding $f$ restricted to $K_{n-2}^{(kl)}$, by the assumption we have 
\begin{eqnarray}\label{l21kn-2pq}
\sum_{\lambda\in\Gamma_{t,t+s}(K_{n-2}^{(kl)})}{\rm lk}(f(\lambda))^2=a\sum_{\lambda\in\Gamma_{3,3}(K_{n-2}^{(kl)})}{\rm lk}(f(\lambda))^2.
\end{eqnarray}
By combining (\ref{l21kn-2}) and (\ref{l21kn-2pq}), we have
\begin{eqnarray}\label{l21kn-2s}
&&\sum_{\substack{\lambda=\gamma\cup\gamma'\in\Gamma_{t+2,t+s}(K_n)\\ \gamma\in\Gamma_{t+2}(K_n),\ \gamma'\in\Gamma_{t+s}(K_n)\\ \overline{kl}\subset\gamma}}{\rm lk}(f(\lambda))^2
+\sum_{\substack{\lambda=\gamma\cup\gamma'\in\Gamma_{t,t+s+2}(K_n)\\ \gamma\in\Gamma_t(K_n),\ \gamma'\in\Gamma_{t+s+2}(K_n)\\ \overline{kl}\subset\gamma'}}{\rm lk}(f(\lambda))^2\\
&=&
a\sum_{\substack{\lambda=\gamma\cup\gamma'\in\Gamma_{3,5}(K_n)\\ \gamma\in\Gamma_5(K_n),\ \gamma'\in\Gamma_3(K_n)\\ \overline{kl}\subset\gamma}}{\rm lk}(f(\lambda))^2
+2a(2t+s-6)\sum_{\lambda\in\Gamma_{3,3}(K_{n-2}^{(kl)})}{\rm lk}(f(\lambda))^2.\nonumber
\end{eqnarray}

Now we assume that $s=1$ or $s\ge 3$. Let us take the sum of both sides of (\ref{l21kn-2s}) over $1\leq k<l\leq n$. 
%1
For a pair of disjoint cycles $\lambda$ of $K_n$ consisting of a $(t+2)$-cycle $\gamma$ and a $(t+s)$-cycle $\gamma'$, let $\overline{kl}$ be an edge of $K_n$ which is contained in $\gamma$. Note that there are $t+2$ ways to choose such an edge $\overline{kl}$. This implies that
\begin{eqnarray}\label{l23t2tl2}
\sum_{1\leq k<l\leq n}
\sum_{\substack{\lambda=\gamma\cup\gamma'\in\Gamma_{t+2,t+s}(K_n)\\ \gamma\in\Gamma_{t+2}(K_n),\ \gamma'\in\Gamma_{t+s}(K_n)\\ \overline{kl}\subset\gamma}}\!\!\!\!\!\!\!\!\!\! {\rm lk}(f(\lambda))^2
=(t+2)\sum_{\lambda\in\Gamma_{t+2,t+s}(K_n)}\!\!\!\!\! {\rm lk}(f(\lambda))^2.
\end{eqnarray}
In the same way as (\ref{l23t2tl2}), we also have 
\begin{eqnarray}
&&\sum_{1\leq k<l\leq n}
\sum_{\substack{\lambda=\gamma\cup\gamma'\in\Gamma_{t,t+s+2}(K_n)\\ \gamma\in\Gamma_t(K_n),\ \gamma'\in\Gamma_{t+s+2}(K_n)\\ \overline{kl}\subset\gamma'}}\!\!\!\!\! {\rm lk}(f(\lambda))^2
=(t+s+2)\sum_{\lambda\in\Gamma_{t,t+s+2}(K_n)}\!\!\!\!\! {\rm lk}(f(\lambda))^2, \label{l23ttl22}\\
&&\sum_{1\leq k<l\leq n}
\sum_{\substack{\lambda=\gamma\cup\gamma'\in\Gamma_{3,5}(K_n)\\ \gamma\in\Gamma_5(K_n),\ \gamma'\in\Gamma_3(K_n)\\ \overline{kl}\subset\gamma}}\!\!\!\!\! {\rm lk}(f(\lambda))^2=5\sum_{\lambda\in\Gamma_{3,5}(K_n)}\!\!\!\!\! {\rm lk}(f(\lambda))^2.\label{l23532}
\end{eqnarray}
%4
For a pair of two disjoint $3$-cycles $\lambda$ of $K_n$, let $k,l$ be two distinct vertices of $K_n$ which are not contained in $\lambda$. Then $\lambda$ is a pair of two disjoint $3$-cycles of $K_{n-2}^{(kl)}$. Note that there are $\binom{n-6}{2}$ ways to choose such two vertices $k,l$. This implies that
\begin{eqnarray}\label{l23332}
\sum_{1\leq k<l\leq n}
\sum_{\lambda\in\Gamma_{3,3}(K_{n-2}^{(kl)})}{\rm lk}(f(\lambda))^2=\binom{n-6}{2}\sum_{\lambda\in\Gamma_{3,3}(K_n)}{\rm lk}(f(\lambda))^2.
\end{eqnarray}
By combining (\ref{l23t2tl2}), (\ref{l23ttl22}), (\ref{l23532}) and (\ref{l23332}) with (\ref{l21kn-2s}), we have
\begin{eqnarray}\label{l23kn}
&&(t+2)\sum_{\lambda\in\Gamma_{t+2,t+s}(K_n)}{\rm lk}(f(\lambda))^2
+(t+s+2)\sum_{\lambda\in\Gamma_{t,t+s+2}(K_n)}{\rm lk}(f(\lambda))^2\\
&=&
5a\sum_{\lambda\in\Gamma_{3,5}(K_n)}{\rm lk}(f(\lambda))^2
+2a(2t+s-6)\binom{n-6}{2}\sum_{\lambda\in\Gamma_{3,3}(K_n)}{\rm lk}(f(\lambda))^2.\nonumber
\end{eqnarray}
Then by (\ref{l23kn}) and Lemma \ref{LL2}, we have (\ref{L23}).

%%%
Next, assume that $s=2$. Then by (\ref{l21kn-2s}), we have
\begin{eqnarray}\label{l22kn-2}
&&\sum_{\substack{\lambda\in\Gamma_{t+2,t+2}(K_n)\\ \overline{kl}\subset\lambda}}{\rm lk}(f(\lambda))^2
+\sum_{\substack{\lambda=\gamma\cup\gamma'\in\Gamma_{t,t+4}(K_n)\\ \gamma\in\Gamma_t(K_n),\ \gamma'\in\Gamma_{t+4}(K_n)\\ \overline{kl}\subset\gamma'}}{\rm lk}(f(\lambda))^2\\
&=&
a\sum_{\substack{\lambda=\gamma\cup\gamma'\in\Gamma_{3,5}(K_n)\\ \gamma\in\Gamma_5(K_n),\ \gamma'\in\Gamma_3(K_n)\\ \overline{kl}\subset\gamma}}{\rm lk}(f(\lambda))^2
+2a(2t-4)\sum_{\lambda\in\Gamma_{3,3}(K_{n-2}^{(kl)})}{\rm lk}(f(\lambda))^2.\nonumber
\end{eqnarray}
Let us take the sum of both sides of (\ref{l22kn-2}) over $1\leq k<l\leq n$. For a pair of disjoint $(t+2)$-cycles $\lambda$ of $K_n$, let $\overline{kl}$ be an edge of $K_n$ which is contained in $\lambda$. Note that there are $2(t+2)$ ways to choose such an edge $\overline{kl}$. This implies that
\begin{eqnarray}\label{l22t2t22}
\sum_{1\leq k<l\leq n}
\sum_{\substack{\lambda\in\Gamma_{t+2,t+2}(K_n)\\ \overline{kl}\subset\lambda}}{\rm lk}(f(\lambda))^2
=2(t+2)\sum_{\lambda\in\Gamma_{t+2,t+2}(K_n)}{\rm lk}(f(\lambda))^2.
\end{eqnarray}
By combining (\ref{l23ttl22}), (\ref{l23532}), (\ref{l23332}) and (\ref{l22t2t22}) with (\ref{l22kn-2}), we have
\begin{eqnarray}\label{l22kn}
&&2(t+2)\sum_{\lambda\in\Gamma_{t+2,t+2}(K_n)}{\rm lk}(f(\lambda))^2
+(t+4)\sum_{\lambda\in\Gamma_{t,t+4}(K_n)}{\rm lk}(f(\lambda))^2\\
&=&
5a\sum_{\lambda\in\Gamma_{3,5}(K_n)}{\rm lk}(f(\lambda))^2
+4a(t-2)\binom{n-6}{2}\sum_{\lambda\in\Gamma_{3,3}(K_n)}{\rm lk}(f(\lambda))^2.\nonumber
\end{eqnarray}
Then by (\ref{l22kn}) and Lemma \ref{LL2}, we have (\ref{L22}).

%%%
Finally, assume that $s=0$. By (\ref{32assume}) we have 
\begin{eqnarray}\label{l20fijpq}
&&\sum_{\substack{\lambda=\gamma\cup\gamma'\in\Gamma_{t+2,t}(F_{ij}^{(kl)})\\ \gamma\in\Gamma_{t+2}(F_{ij}^{(kl)}),\ \gamma'\in\Gamma_t(F_{ij}^{(kl)}) \\ \overline{iklj}\subset\gamma}}{\rm lk}(f(\lambda))^2
+\sum_{\substack{\lambda\in\Gamma_{t,t}(F_{ij}^{(kl)})\\ \overline{iklj}\not\subset\lambda}}{\rm lk}(f(\lambda))^2\\
&=&a\Bigg( \sum_{\substack{\lambda=\gamma\cup\gamma'\in\Gamma_{3,5}(F_{ij}^{(kl)})\\ \gamma\in\Gamma_5(F_{ij}^{(kl)}),\ \gamma'\in\Gamma_3(F_{ij}^{(kl)})\\ \overline{iklj}\subset\gamma}}{\rm lk}(f(\lambda))^2
+\sum_{\substack{\lambda\in\Gamma_{3,3}(F_{ij}^{(kl)})\\ \overline{iklj}\not\subset\lambda}}{\rm lk}(f(\lambda))^2\Bigg). \nonumber
\end{eqnarray}
Note that, for the embedding $f$ restricted to $F_{ij}^{(lk)}$, we have a similar formula as (\ref{l20fijpq}). 
%
%\begin{eqnarray}\label{l20fijqp}
%&&\sum_{\substack{\lambda=\gamma\cup\gamma'\in\Gamma_{t+2,t}(F_{ij}^{(qp)})\\ \gamma\in\Gamma_{t+2}(F_{ij}^{(qp)}), \gamma'\in\Gamma_t(F_{ij}^{(qp)}) \\ e_{qp}\subset\gamma}}{\rm lk}(f(\lambda))^2
%+\sum_{\substack{\lambda\in\Gamma_{t,t}(F_{ij}^{(qp)})\\ e_{qp}\not\subset\lambda}}{\rm lk}(f(\lambda))^2\\
%&=&a\Bigg( \sum_{\substack{\lambda=\gamma\cup\gamma'\in\Gamma_{5,3}(F_{ij}^{(qp)})\\ \gamma\in\Gamma_5(F_{ij}^{(qp)}), \gamma'\in\Gamma_3(F_{ij}^{(qp)})\\ e_{qp}\subset\gamma}}{\rm lk}(f(\lambda))^2
%+\sum_{\substack{\lambda\in\Gamma_{3,3}(F_{ij}^{(qp)})\\ e_{qp}\not\subset\lambda}}{\rm lk}(f(\lambda))^2\Bigg).\nonumber
%\end{eqnarray}
%
Let us take the sum of both side of (\ref{l20fijpq}) over $1\leq i<j\leq n$ and $i, j\neq k, l$. 
By combining (\ref{l21t2tl}), (\ref{l2153}), (\ref{l21ttlpq}), (\ref{l21ttlqp}), (\ref{l2133pq}) and (\ref{l2133qp}) with (\ref{l20fijpq}), we have
\begin{eqnarray}\label{l20kn-2}
&&\sum_{\substack{\lambda=\gamma\cup\gamma'\in\Gamma_{t+2,t}(K_n)\\ \gamma\in\Gamma_{t+2}(K_n),\ \gamma'\in\Gamma_t(K_n)\\ \overline{kl}\subset\gamma}}{\rm lk}(f(\lambda))^2
+2\left( \binom{n-2}{2}-2t \right)\sum_{\lambda\in\Gamma_{t,t}(K_{n-2}^{(kl)})}{\rm lk}(f(\lambda))^2\\
&=&
a\sum_{\substack{\lambda=\gamma\cup\gamma'\in\Gamma_{3,5}(K_n)\\ \gamma\in\Gamma_5(K_n),\ \gamma'\in\Gamma_3(K_n)\\ \overline{kl}\subset\gamma}}{\rm lk}(f(\lambda))^2
+2a\left( \binom{n-2}{2}-6 \right)\sum_{\lambda\in\Gamma_{3,3}(K_{n-2}^{(kl)})}{\rm lk}(f(\lambda))^2.\nonumber
\end{eqnarray}
Then for the embedding $f$ restricted to $K_{n-2}^{(kl)}$, by the assumption we have 
\begin{eqnarray}\label{l20kn-2pq}
\sum_{\lambda\in\Gamma_{t,t}(K_{n-2}^{(kl)})}{\rm lk}(\lambda))^2=a\sum_{\lambda\in\Gamma_{3,3}(K_{n-2}^{(kl)})}{\rm lk}(f(\lambda))^2.
\end{eqnarray}
By combining (\ref{l20kn-2}) and (\ref{l20kn-2pq}), we have
\begin{eqnarray}\label{l20kn-2s}
&&\sum_{\substack{\lambda=\gamma\cup\gamma'\in\Gamma_{t+2,t}(K_n)\\ \gamma\in\Gamma_{t+2}(K_n),\ \gamma'\in\Gamma_t(K_n)\\ \overline{kl}\subset\gamma}}{\rm lk}(f(\lambda))^2\\
&=&
a\sum_{\substack{\lambda=\gamma\cup\gamma'\in\Gamma_{3,5}(K_n)\\ \gamma\in\Gamma_5(K_n),\ \gamma'\in\Gamma_3(K_n)\\ \overline{kl}\subset\gamma}}{\rm lk}(f(\lambda))^2
+2a(2t-6)\sum_{\lambda\in\Gamma_{3,3}(K_{n-2}^{(kl)})}{\rm lk}(f(\lambda))^2.\nonumber
\end{eqnarray}
Now we take the sum of both sides of (\ref{l20kn-2s}) over $1\leq k<l\leq n$. By combining (\ref{l23t2tl2}), (\ref{l23532}) and (\ref{l23332}) with (\ref{l20kn-2s}), we have
\begin{eqnarray}\label{l20kn}
&&(t+2)\sum_{\lambda\in\Gamma_{t+2,t}(K_n)}{\rm lk}(f(\lambda))^2\\
&=&5a\sum_{\lambda\in\Gamma_{3,5}(K_n)}{\rm lk}(f(\lambda))^2
+4a(t-3)\binom{n-6}{2}\sum_{\lambda\in\Gamma_{3,3}(K_n)}{\rm lk}(f(\lambda))^2.\nonumber
\end{eqnarray}
Then by (\ref{l20kn}) and Lemma \ref{LL2}, we have (\ref{L20}).
\end{proof}

\begin{proof}[Proof of Theorem \ref{lkrefine}]
First we show in the case of $p=q=n/2$ by induction on $p$. If $p=4$, then by Theorem \ref{4433}, we have the result. 
Assume that $p\ge 5$ and it holds that
\begin{eqnarray}\label{ka1}
\sum_{\lambda\in\Gamma_{p-1,p-1}(K_{2p-2})}{\rm lk}(g(\lambda))^2=(2p-8)!\sum_{\lambda\in\Gamma_{3,3}(K_{2p-2})}{\rm lk}(g(\lambda))^2 
\end{eqnarray}
for any spatial embedding $g$ of $K_{2p-2}$. 
Then by (\ref{ka1}) and setting $n=2p-1$, $a=(2p-8)!$, $t=p-1$ and $s=0$ in Lemma \ref{L1}, we have 
\begin{eqnarray}\label{ka11}
\sum_{\lambda\in\Gamma_{p-1,p}(K_{2p-1})}{\rm lk}(h(\lambda))^2=2\cdot{(2p-7)!}\sum_{\lambda\in\Gamma_{3,3}(K_{2p-1})}{\rm lk}(h(\lambda))^2 
\end{eqnarray}
for any spatial embedding $h$ of $K_{2p-1}$. Then by (\ref{ka11}) and setting $n=2p$, $a=2\cdot (2p-7)!$, $t=p-1$ and $s=1$ in Lemma \ref{L1}, we have 
\begin{eqnarray}\label{ka12}
&&2p\sum_{\lambda\in\Gamma_{p,p}(K_{2p})}{\rm lk}(f(\lambda))^2+(p+1)\sum_{\lambda\in\Gamma_{p-1,p+1}(K_{2p})}{\rm lk}(f(\lambda))^2\\
&=&
2\cdot{(2p-6)!}\cdot (2p+1)\sum_{\lambda\in\Gamma_{3,3}(K_{2p})}{\rm lk}(f(\lambda))^2\nonumber
\end{eqnarray}
for any spatial embedding $f$ of $K_{2p}$. On the other hand, by (\ref{ka1}) and setting $n=2p$, $a=(2p-8)!$, $t=p-1$ and $s=0$ in Lemma \ref{L2}, we also have 
\begin{eqnarray}\label{ka13}
\sum_{\lambda\in\Gamma_{p-1,p+1}(K_{2p})}{\rm lk}(f(\lambda))^2=2\cdot{(2p-6)!}\sum_{\lambda\in\Gamma_{3,3}(K_{2p})}{\rm lk}(f(\lambda))^2.
\end{eqnarray}
Then by (\ref{ka12}) and (\ref{ka13}), we have the result. 

Next, we show in the case of $n=p+q\ (p<q)$ by the induction on $r=q-p\ge 1$. 
In the case of $r=1$, as we showed in the first half of this proof, we have 
\begin{eqnarray}\label{2p}
\sum_{\lambda\in\Gamma_{p,p}(K_{2p})}{\rm lk}(g(\lambda))^2=(2p-6)!\sum_{\lambda\in\Gamma_{3,3}(K_{2p})}{\rm lk}(g(\lambda))^2
\end{eqnarray}
for any spatial embedding $g$ of $K_{2p}$. Then by (\ref{2p}) and setting $n=2p+1$, $a=(2p-6)!$, $t=p$ and $s=0$ in Lemma \ref{L1}, we have 
\begin{eqnarray*}
\sum_{\lambda\in\Gamma_{p,p+1}(K_{2p+1})}{\rm lk}(f(\lambda))^2=2\cdot{(2p-5)!}\sum_{\lambda\in\Gamma_{3,3}(K_{2p+1})}{\rm lk}(f(\lambda))^2
\end{eqnarray*}
for any spatial embedding $f$ of $K_{2p+1}$. Thus we have the result. 

In the case of $r=2$, by (\ref{2p}) and setting  $n=2p+2$, $a=(2p-6)!$, $t=p$ and $s=0$ in Lemma \ref{L2}, we have 
\begin{eqnarray*}
\sum_{\lambda\in\Gamma_{p,p+2}(K_{2p+2})}{\rm lk}(f(\lambda))^2=2\cdot{(2p-4)!}\sum_{\lambda\in\Gamma_{3,3}(K_{2p+2})}{\rm lk}(f(\lambda))^2 
\end{eqnarray*}
for any spatial embedding $f$ of $K_{2p+1}$. Thus we have the result.

Assume that $r\ge3$ and it holds that 
\begin{eqnarray}\label{ka2}
\sum_{\lambda\in\Gamma_{p,p+(r-1)}(K_{2p+r-1})}{\rm lk}(g(\lambda))^2=2\cdot{(2p+r-7)!}\sum_{\lambda\in\Gamma_{3,3}(K_{2p+r-1})}{\rm lk}(g(\lambda))^2
\end{eqnarray}
for any spatial embedding $g$ of $K_{2p+r-1}$. Then by (\ref{ka2}) and setting  $n=2p+r$, $a=2\cdot{(2p+r-7)!}$, $t=p$ and $s=r-1\ge 2$ in Lemma \ref{L1}, we have 
\begin{eqnarray}\label{ka21}
&&(p+1)\sum_{\lambda\in\Gamma_{p+1,p+(r-1)}(K_{2p+r})}{\rm lk}(f(\lambda))^2+(p+r)\sum_{\lambda\in\Gamma_{p,p+r}(K_{2p+r})}{\rm lk}(f(\lambda))^2\\
&=&2\cdot{(2p+r-6)!}\cdot (2p+r+1)\sum_{\lambda\in\Gamma_{3,3}(K_{2p+r})}{\rm lk}(f(\lambda))^2.\nonumber
\end{eqnarray}
Here, by the induction hypothesis, we also have 
\begin{eqnarray}\label{ka22}
\sum_{\lambda\in\Gamma_{p+1,p+(r-1)}(K_{2p+r})}{\rm lk}(f(\lambda))^2
&=&
\sum_{\lambda\in\Gamma_{p+1,(p+1)+(r-2)}(K_{2p+r})}{\rm lk}(f(\lambda))^2\\
&=&2\cdot{(2p+r-6)!}\sum_{\lambda\in\Gamma_{3,3}(K_{2p+r})}{\rm lk}(f(\lambda))^2.\nonumber
\end{eqnarray}
By (\ref{ka21}) and (\ref{ka22}), we have 
\begin{eqnarray}\label{pqfin}
&&2(p+1)\cdot{(2p+r-6)!}\sum_{\lambda\in\Gamma_{3,3}(K_{2p+r})}\!\!\!\!\! {\rm lk}(f(\lambda))^2
+(p+r)\sum_{\lambda\in\Gamma_{p,p+r}(K_{2p+r})}\!\!\!\!\! {\rm lk}(f(\lambda))^2\\
&=&2\cdot{(2p+r-6)!}\cdot (2p+r+1)\sum_{\lambda\in\Gamma_{3,3}(K_{2p+r})}\!\!\!\!\! {\rm lk}(f(\lambda))^2.\nonumber
\end{eqnarray}
By (\ref{pqfin}), we have the result.

Finally we show (\ref{ilKK}).  
Assume that $n$ is odd. Then there are $(n-5)/2$ ways to choose a pair of two integers $p,q\ge 3$ with $n=p+q$. Then by (\ref{ilKK0}), we have 
\begin{eqnarray*}
\sum_{p+q=n}\sum_{\lambda\in\Gamma_{p,q}(K_n)}{\rm lk}(f(\lambda))^2 
&=&\frac{n-5}{2}\cdot 2\cdot{(n-6)!}\sum_{\lambda\in\Gamma_{3,3}(K_n)}{\rm lk}(f(\lambda))^2\\
&=&(n-5)!\sum_{\lambda\in\Gamma_{3,3}(K_n)}{\rm lk}(f(\lambda))^2.
\end{eqnarray*}
On the other hand, assume that $n$ is even. Then there are $(n-6)/2$ ways to choose a pair of two integers $p,q\ge 3$ with $p\neq q$ and $n=p+q$, and there is only one way to choose a pair of two integers $p,q$ with $p=q$ and $n=p+q$ (namely $p=q=n/2)$). Then by (\ref{ilKK0}), we have 
\begin{eqnarray*}
&&\sum_{p+q=n}\sum_{\lambda\in\Gamma_{p,q}(K_n)}{\rm lk}(f(\lambda))^2\\
&=&\sum_{\substack{p+q=n\\ p\neq q}}\sum_{\lambda\in\Gamma_{p,q}(K_n)}{\rm lk}(f(\lambda))^2
+\sum_{\lambda\in\Gamma_{\frac{n}{2},\frac{n}{2}}(K_n)}{\rm lk}(f(\lambda))^2\\
&=&\frac{n-6}{2}\cdot 2\cdot{(n-6)!}\sum_{\lambda\in\Gamma_{3,3}(K_n)}{\rm lk}(f(\lambda))^2+(n-6)!\sum_{\lambda\in\Gamma_{3,3}(K_n)}{\rm lk}(f(\lambda))^2\\
&=&(n-5)!\sum_{\lambda\in\Gamma_{3,3}(K_n)}{\rm lk}(f(\lambda))^2.
\end{eqnarray*}
This completes the proof. 
\end{proof}

\begin{proof}[Proof of Corollary \ref{maincor1}]
We show (1). For any two spatial embeddings $f$ and $g$ of $K_{n}$, by (\ref{ilKK0}), we have 
\begin{eqnarray}\label{co1}
&& \sum_{\lambda\in\Gamma_{p,p}(K_n)}{\rm lk}(f(\lambda))^2
- \sum_{\lambda\in\Gamma_{p,p}(K_n)}{\rm lk}(g(\lambda))^2 \\ 
&=& 
(n-6)!
\Big(\sum_{\lambda\in \Gamma_{3,3}\left(K_{n}\right)}{{\rm lk}\left(f(\lambda)\right)}^{2}
- \sum_{\lambda\in \Gamma_{3,3}\left(K_{n}\right)}{{\rm lk}\left(g(\lambda)\right)}^{2}
\Big). \nonumber
\end{eqnarray}
Since $\sum_{\lambda\in \Gamma_{3,3}\left(K_{n}\right)}{{\rm lk}\left(f(\lambda)\right)}^{2}$ and $\sum_{\lambda\in \Gamma_{3,3}\left(K_{n}\right)}{{\rm lk}\left(g(\lambda)\right)}^{2}$ have the same parity, that is also equal to the parity of $\binom{n}{6}$, by (\ref{co1}), we have 
\begin{eqnarray}\label{co2}
\sum_{\lambda\in\Gamma_{p,p}(K_n)}{\rm lk}(f(\lambda))^2
\equiv  \sum_{\lambda\in\Gamma_{p,p}(K_n)}{\rm lk}(g(\lambda))^2 
\pmod{2(n-6)!}. 
\end{eqnarray}
Note that there exists a spatial embedding $g$ of $K_{n}$ such that 
\begin{eqnarray}\label{co4}
\sum_{\lambda\in \Gamma_{3,3}\left(K_{n}\right)}{{\rm lk}\left(g(\lambda)\right)}^{2} = \binom{n}{6}, 
\end{eqnarray}
see Remark \ref{ineq_sharp}. Thus by (\ref{co2}) and (\ref{co4}), we have 
\begin{eqnarray*}\label{co5}
\sum_{\lambda\in \Gamma_{p,p}\left(K_{n}\right)}{{\rm lk}\left(f(\lambda)\right)}^{2} \equiv {(n-6)!}\binom{n}{6}\pmod{2(n-6)!}
\end{eqnarray*}
for any spatial embedding $f$ of $K_{n}$. Since $\binom{n}{6}$ is odd if and only if $n\equiv 6,7\pmod{8}$, we have the result. (2) and (3) can be shown in the same way as (1). 
\end{proof}

\begin{proof}[Proof of Corollary \ref{Cor9}]
Note that  no pair of two disjoint $3$-cycles $\lambda$ of $K_{n}$ is shared by two  distinct subgraphs of $K_{n}$ isomorphic to $K_{6}$. Then for any spatial embedding $f$ of $K_{n}$, Theorem \ref{CG} (1) implies that 
\begin{eqnarray}\label{lb}
\sum_{\lambda\in \Gamma_{3,3}\left(K_{n}\right)}{{\rm lk}\left(f(\lambda)\right)}^{2}\ge \binom{n}{6}. 
\end{eqnarray}
Thus by (\ref{lb}) and Theorem \ref{lkrefine}, we have the result. 
\end{proof}

\begin{Remark}\label{ineq_sharp}
As it was pointed out in \cite[Remark 2.5]{MN19}, the lower bound of (\ref{lb}) is realized by a {\it canonical book presentation} \cite{EO94} of $K_{n}$, which contains exactly $\binom{n}{6}$ Hopf links corresponding to all the pairs of two disjoint $3$-cycles of $K_{n}$ if $n\ge 6$ \cite{Otsuki96} (see also Example \ref{k8twoexs} in the case of $n=8$). Thus the lower bound of each of the inequalities in Corollary \ref{Cor9} is sharp. 
\end{Remark}

\begin{Example}\label{k7twoexs}
Let $g$ and $h$ be two spatial embeddings of $K_{7}$ as illustrated in Fig. \ref{K7link_examples}. The embedding $g$ was given in \cite{CG83} as an embedding that the image contains exactly one nontrivial knot, and the embedding $h$ is obtained from $g$ by a single crossing change at the crossing between $g(\overline{24})$ and $g(\overline{35})$. Then we can see that all of the nonsplittable constituent $2$-component links of type $(3,3)$ of $g(K_{7})$ are exactly $7$ Hopf links, and the ones of $h(K_{7})$ are exactly $9$ Hopf links. Thus by Theorem \ref{lkrefine} ((\ref{k73433})) we have 
\begin{eqnarray}
&& \sum_{\lambda\in \Gamma_{3,4}(K_{7})}{\rm lk}(g(\lambda))^{2} = 14, \label{k7ex14}\\
&& \sum_{\lambda\in \Gamma_{3,4}(K_{7})}{\rm lk}(h(\lambda))^{2} = 18.  \label{k7ex18}
\end{eqnarray}
Actually $g(K_{7})$ contains exactly $14$ Hopf links and $h(K_{7})$ 
contains exactly $18$ Hopf links as all of the nonsplittable constituent Hamiltonian $2$-component links. 
Note that for any spatial embedding $f$ of $K_{7}$, by Corollary \ref{maincor1} (2), we have 
\begin{eqnarray}\label{k7lkcongr}
\sum_{\lambda\in \Gamma_{3,4}(K_{7})}{\rm lk}(f(\lambda))^{2}
\equiv 2\pmod{4}. 
\end{eqnarray}
By (\ref{k7ex14}) and (\ref{k7ex18}), the congruence (\ref{k7lkcongr}) is the best possible. 
\end{Example}

\begin{figure}[htbp]
      \begin{center}
            \scalebox{0.525}{\includegraphics*{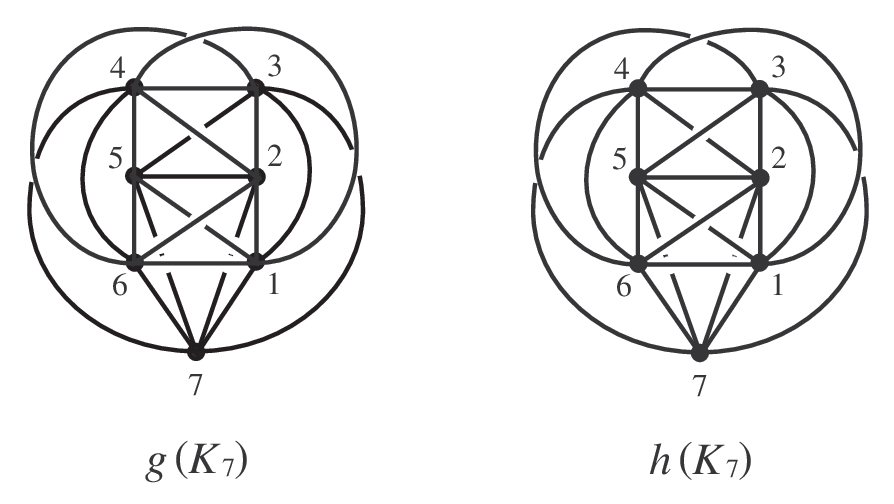}}
%\scalebox{0.675}{\includegraphics[bb = 91 275 520 517, width=12cm,clip]{K7link_examples.pdf}}
      \end{center}
   \caption{Two spatial graphs of $K_{7}$}
  \label{K7link_examples}
\end{figure} 

\begin{Example}\label{k8twoexs}
Let $g$ and $h$ be two spatial embeddings of $K_{8}$ as illustrated in Fig. \ref{K8link_examples}. The embedding $g$ is a canonical book presentation of $K_{8}$ (see Remark \ref{ineq_sharp}), and the embedding $h$ was given in \cite{FM09} as an embedding that the image contains only Hopf links as nonsplittable constituent $2$-component links. 
Then we can see that all of the nonsplittable constituent $2$-component links of type $(3,3)$ of $g(K_{8})$ are exactly $28$ Hopf links, and the ones of $h(K_{8})$ are exactly $30$ Hopf links. Thus by Theorem \ref{lkrefine} (Theorem \ref{lkthm}) we have 
\begin{eqnarray}
&& \sum_{\lambda\in \Gamma_{3,5}(K_{8})}{\rm lk}(g(\lambda))^{2} 
= 2\sum_{\lambda\in \Gamma_{4,4}(K_{8})}{\rm lk}(g(\lambda))^{2}
= 112, \label{k8ex28}\\
&& \sum_{\lambda\in \Gamma_{3,5}(K_{8})}{\rm lk}(h(\lambda))^{2} 
= 2\sum_{\lambda\in \Gamma_{4,4}(K_{8})}{\rm lk}(h(\lambda))^{2}
= 120. \label{k8ex30}
\end{eqnarray}
Actually $g(K_{8})$ contains exactly $112$ Hopf links of type $(3,5)$ and exactly $52$ Hopf links and one $(2,4)$-torus link $g([1357]\cup [2468])$ of type $(4,4)$ as all of the nonsplittable constituent Hamiltonian $2$-component links, where $[i_{1}i_{2}\cdots i_{p}]$ denotes a $p$-cycle $\overline{i_{1}i_{2}}\cup \overline{i_{2}i_{3}}\cup \cdots\cup \overline{i_{p}i_{1}}$, and $h(K_{8})$ contains exactly $120$ Hopf links of type $(3,5)$ and exactly $60$ Hopf links of type $(4,4)$ as all of the nonsplittable constituent Hamiltonian $2$-component links \cite{FM09}. Here we checked the number of nonsplittable links by a computer program {\it Gordian} \cite{gordian}. Note that for any spatial embedding $f$ of $K_{8}$, by Corollary \ref{maincor1} (1) and (2), we have 
\begin{eqnarray}
&& \sum_{\lambda\in \Gamma_{3,5}(K_{8})}{\rm lk}(f(\lambda))^{2}
\equiv 0\pmod{8},\label{k8lkcongr1} \\
&& \sum_{\lambda\in \Gamma_{4,4}(K_{8})}{\rm lk}(f(\lambda))^{2}
\equiv 0\pmod{4}. \label{k8lkcongr2}
\end{eqnarray}
By (\ref{k8ex28}) and (\ref{k8ex30}), both the congruences (\ref{k8lkcongr1}) and (\ref{k8lkcongr2}) are the best possible. 
\end{Example}

\begin{figure}[htbp]
      \begin{center}
      \scalebox{0.5}{\includegraphics*{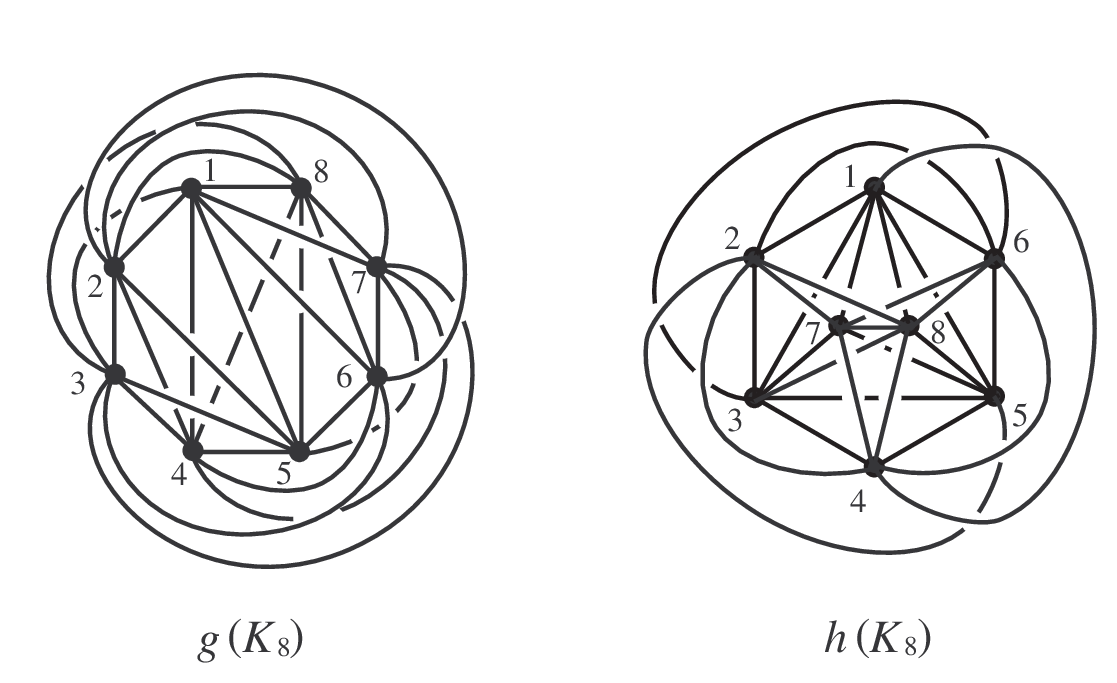}}
%\scalebox{0.7}{\includegraphics[bb = 38 232 574 560, width=14cm,clip]{K8link_examples.pdf}}
      \end{center}
   \caption{Two spatial graphs of $K_{8}$}
  \label{K8link_examples}
\end{figure} 

We strongly believe that all of the congruences in Corollary \ref{maincor1} are the best possible for each $n\ge 6$. To show it, it is sufficient to give the affirmative answer to the following question. 

\begin{Question}
For each integer $n\ge 6$, does there exist a spatial embedding $f$ of $K_{n}$ satisfying $\sum_{\lambda\in \Gamma_{3,3}(K_{n})}{\rm lk}(f(\lambda))^{2} = \binom{n}{6}+2$?  
\end{Question}

\begin{Remark}\label{minnumlk8}
The spatial graph $g(K_{8})$ in Example \ref{k8twoexs} also satisfies 
\begin{eqnarray}
\sum_{\lambda\in \Gamma_{3,4}(K_{8})}{\rm lk}(g(\lambda)) = 112 
\end{eqnarray}
by Lemma \ref{kn3433}, and actually $g(K_{8})$ contains exactly $112$ Hopf links as all of the nonsplittable constituent $2$-component links of type $(3,4)$. Therefore the number of nonsplittable constituent $2$-component links of $g(K_{8})$ is $28+112+53+112 = 305$. This gives an upper bound of the minimum number of nonsplittable constituent $2$-component links of a spatial graph of $K_{8}$ (see \cite[Theorem 2]{FM09}). On the other hand, if a spatial graph $f(K_{8})$ contains only $2$-component links with $\left|{\rm lk}\right|\le 1$ as the nonsplittable constituent $2$-component links, then it follows from  (\ref{lb}), Lemma \ref{kn3433} and Corollary \ref{Cor9} that the number of nonsplittable constituent $2$-component links of $f(K_{8})$ is greater than or equal to $28+112+56+112=308$. This implies that if a spatial graph of $K_{8}$ realizes the minimum number of the nonsplittable constituent $2$-component links then it must contain a $2$-component link with $\left|{\rm lk}\right|\ge 2$. 
\end{Remark}

\begin{proof}[Proof of Corollary \ref{maxlkcor}]
Note that the number of all elements in $\Gamma_{p,q}(K_{n})$ is  
\begin{eqnarray*}
\left\{
   \begin{array}{@{\,}lll}
   {\displaystyle \frac{1}{2}\binom{n}{p}\cdot\frac{(p-1)!}{2}\cdot\frac{(p-1)!}{2}
= \frac{n!}{8p^{2}}} & (p=q) \\
   {\displaystyle \binom{n}{p}\cdot\frac{(p-1)!}{2}\cdot\frac{(q-1)!}{2}
= \frac{n!}{4pq}}  & (p\neq q). 
   \end{array}
\right. 
\end{eqnarray*}
Then by Corollary \ref{Cor9}, if $p=q$ then we have 
\begin{eqnarray}\label{lkpp}
\Big(\mathop{\rm max}\limits_{\lambda\in \Gamma_{p,p}(K_{n})}
\left|{\rm lk}(f(\lambda))\right|\Big)^{2}
\cdot\frac{n!}{8p^{2}}
\ge \sum_{\lambda\in \Gamma_{p,p}(K_{n})}{\rm lk}(f(\lambda))^{2}
\ge \frac{n!}{6!}, 
\end{eqnarray}
and if $p\neq q$ then we have 
\begin{eqnarray}\label{lkpq}
\Big(\mathop{\rm max}\limits_{\lambda\in \Gamma_{p,q}(K_{n})}
\left|{\rm lk}(f(\lambda))\right|\Big)^{2}
\cdot \frac{n!}{4pq}
\ge \sum_{\lambda\in \Gamma_{p,q}(K_{n})}{\rm lk}(f(\lambda))^{2}
\ge 2\cdot \frac{n!}{6!}. 
\end{eqnarray}
Therefore by (\ref{lkpp}) and (\ref{lkpq}), we have 
\begin{eqnarray*}
\Big(\mathop{\rm max}\limits_{\lambda\in \Gamma_{p,q}(K_{n})}
\left|{\rm lk}(f(\lambda))\right|\Big)^{2}
\ge \frac{pq}{90}. 
\end{eqnarray*}
Thus we have the result. 
\end{proof}

\begin{proof}[Proof of Corollary \ref{lkmaxham}]
If $pq > 90(m-1)^{2}$, then by Corollary \ref{maxlkcor} we have 
\begin{eqnarray*}
\mathop{\rm max}\limits_{\lambda\in \Gamma_{p,q}(K_{n})}
\left|{\rm lk}(f(\lambda))\right|
\ge \frac{\sqrt{pq}}{3\sqrt{10}}
> m-1. 
\end{eqnarray*}
This implies the desired conclusion. 
\end{proof}

\section{Applications to rectilinear spatial complete graphs} 

A spatial embedding $f_{\rm r}$ of a simple graph $G$ is said to be {\it rectilinear} if for any edge $e$ of $G$, $f_{\rm r}(e)$ is a straight line segment in ${\mathbb R}^{3}$. As we mentioned in Section $1$, the rectilinear spatial graph serves as a mathematical model for chemical compounds. Thus from the viewpoint of application to molecular topology, we are interested in the behavior of the nontrivial knots and links in rectilinear spatial graphs. We refer the reader to \cite{B77}, \cite{RA99}, \cite{Hughes06}, \cite{HJ07}, \cite[\S 4]{Nikkuni09}, \cite{Huh12}, \cite[\S 4]{HN14}, \cite{RE14}, \cite{RE15}, \cite[\S 6]{FMMRN17} for works on knots and links in rectilinear spatial graphs, and \cite{FK16}, \cite{FKN19} for works on random rectilinear spatial graphs. See also \cite{MN19} for intrinsic knotting on rectilinear spatial graphs of $K_{n}$ revealed by Theorem \ref{mainthm}. 

Now let us observe intrinsic linking on rectilinear spatial graphs of $K_{n}$ based on Theorem \ref{lkrefine}. Note that every constituent link of a rectilinear spatial graph is a polygonal link with some sticks, and it is well-known that every polygonal $2$-component link with exactly six sticks is a trivial link or a Hopf link (unlinked two triangles or linked two triangles). Then it follows from this fact that the number of ``triangle-triangle'' Hopf links in a rectilinear spatial graph coincides with the sum of ${\rm lk}^{2}$ over all of the constituent triangle-triangle links. Therefore Theorem \ref{lkrefine} implies that for a rectilinear spatial graph of $K_{n}$, the sum of ${\rm lk}^{2}$ over all of the constituent $2$-component Hamiltonian links of any type $(p,q)$ is determined explicitly in terms of the number of triangle-triangle Hopf links. Then we recall the fact that every rectilinear spatial graph of $K_{6}$ contains at most three Hopf links \cite{Hughes06}, \cite{HJ07}, \cite{Nikkuni09}. This also implies that the number of triangle-triangle Hopf links in a rectilinear spatial graph of $K_{n}$ is less than or equal to $3\binom{n}{6}$. Thus by Theorem \ref{lkrefine} and Corollary \ref{Cor9}, we have the following.

\begin{Corollary}\label{Cor92}
Let $n\ge 6$ be an integer and $p,q\ge 3$ two integers satisfying $n=p+q$. For any rectilinear spatial embedding $f_{\rm r}$ of $K_n$, we have 
\begin{eqnarray}\label{ieMN02a}
\frac{n!}{6!} \le \sum_{\lambda\in\Gamma_{p,q}(K_n)}{\rm lk}(f_{\rm r}(\lambda))^2
\le 3\cdot \frac{n!}{6!}
\end{eqnarray}
if $p=q$, and 
\begin{eqnarray}\label{ieMN02b}
2\cdot \frac{n!}{6!} \le \sum_{\lambda\in\Gamma_{p,q}(K_n)}{\rm lk}(f_{\rm r}(\lambda))^2
\le 6\cdot \frac{n!}{6!}
\end{eqnarray}
if $p\neq q$. In particular, we have 
\begin{eqnarray}\label{ieMN2}
(n-5)\cdot \frac{n!}{6!}
\le 
\sum_{p+q=n}
\sum_{\lambda\in\Gamma_{p,q}(K_n)}{\rm lk}(f_{\rm r}(\lambda))^2 
\le
3(n-5)\cdot \frac{n!}{6!}.
\end{eqnarray}
\end{Corollary}

\begin{Remark}\label{ineq_sharp_recti}
As it was pointed out in \cite[Remark 2.7]{MN19}, the lower bound of (\ref{lb}) is also realized by ``standard'' rectilinear spatial embedding of $K_{n}$. This implies that the lower bound of each of the inequalities in Corollary \ref{Cor92} is sharp. But we do not think that the upper bound is also sharp if $n\ge 7$, see Example \ref{ub2}. 
\end{Remark}

\begin{Example}\label{ub2}
Let $f_{\rm r}$ be a rectilinear spatial embedding of $K_{7}$. Then by (\ref{ieMN02b}) and (\ref{k7lkcongr}), we have 
\begin{eqnarray*}\label{k7lkv}
14\le \sum_{\lambda\in \Gamma_{3,4}(K_{7})}{\rm lk}(f_{\rm r}(\lambda))^{2}\le 42,\ \ \sum_{\lambda\in \Gamma_{3,4}(K_{7})}{\rm lk}(f_{\rm r}(\lambda))^{2}\equiv 2\pmod{4}. 
\end{eqnarray*}
However, according to a computer search in Jeon et al. \cite{Jeon10}, there seems to be no rectilinear  embedding $f_{\rm r}$ of $K_{7}$ such that $\sum_{\lambda\in \Gamma_{3,4}(K_{7})}{\rm lk}(f_{\rm r}(\lambda))^{2} = 38, 42$, or equivalently by (\ref{ilKK0}), $\sum_{\lambda\in \Gamma_{3,3}(K_{7})}{\rm lk}(f_{\rm r}(\lambda))^{2} = 19,21$. This suggests that the upper bound in Corollary \ref{Cor92} cannot be expected to be sharp if $n\ge 7$. 
\end{Example}

The problem of determining the sharp upper bound of the sum of ${\rm lk}^{2}$ over all of the constituent $2$-component Hamiltonian links of each type for all rectilinear spatial graphs of $K_{n}$ is equivalent to the following problem, that is also equvalent to a problem which has already been stated by the authors in \cite[Problem 3.4]{MN19}. 

\begin{Problem}
Determine the maximum number of constituent triangle-triangle Hopf links for all rectilinear spatial graphs of $K_{n}$ for each $n\ge 7$. 
\end{Problem}

Finally, let us show similar results as Corollaries \ref{maxlkcor} and \ref{lkmaxham} for the maximum value of $a_{2}$ over all of the Hamiltonian knots of a rectilinear spatial graph of $K_{n}$. It is also well-known that every polygonal knot with less than or equal to five sticks is trivial. Then by combining this fact and (\ref{lb}) with Theorem \ref{mainthm}, we have 
\begin{eqnarray}\label{maintheoremcor}
\sum_{\gamma\in \Gamma_{n}(K_{n})}a_{2}(f_{\rm r}(\gamma))
\ge  
\frac{(n-5)(n-6)\cdot (n-1)!}{2\cdot 6!} 
\end{eqnarray}
for any rectilinear spatial embedding $f_{\rm r}$ of $K_{n}\ (n\ge 6)$ \cite[Corollary 1.7]{MN19}. Then we have the following. 

\begin{Corollary}\label{maxa2cor}
Let $n\ge 6$ be an integer. For any rectilinear spatial embedding $f_{\rm r}$ of $K_n$, we have 
\begin{eqnarray*}\label{maxa2}
\mathop{\rm max}\limits_{\gamma\in \Gamma_{n}(K_{n})}
a_{2}(f_{\rm r}(\gamma))
\ge
\frac{(n-5)(n-6)}{6!}. 
\end{eqnarray*}
\end{Corollary}

\begin{proof}
Note that the number of all elements in $\Gamma_{n}(K_{n})$ is $(n-1)!/2$. 
Then by (\ref{maintheoremcor}), we have 
\begin{eqnarray*}\label{a2n}
\mathop{\rm max}\limits_{\lambda\in \Gamma_{n}(K_{n})}
a_{2}(f_{\rm r}(\gamma))
\cdot \frac{(n-1)!}{2}
\ge \sum_{\lambda\in \Gamma_{n}(K_{n})}a_{2}(f_{\rm r}(\gamma))
\ge \frac{(n-5)(n-6)\cdot (n-1)!}{2\cdot 6!}.
\end{eqnarray*}
Thus we have the result. 
\end{proof}

Corollary \ref{maxa2cor} says that if $n$ is sufficiently large then every rectilinear spatial graph of $K_{n}$ contains a Hamiltonian knot whose value of $a_{2}$ is arbitrary large. Actually we have the following. 

\begin{Corollary}\label{a2maxham}
Let $n\ge 6$ be a positive integer. For a rectilinear spatial embedding $f_{\rm r}$ of $K_{n}$ and a positive integer $m$, if $n> (11+\sqrt{2880m-2879})/2$ then there exists a cycle $\gamma\in \Gamma_{n}(K_{n})$ such that $a_{2}(f_{\rm r}(\gamma))\ge m$. 
\end{Corollary}

\begin{proof}
For an integer $n\ge 6$ and a positive integer $m$, we can see that $n> (11+\sqrt{2880m-2879})/2$ if and only if $(n-5)(n-6)/6! > m-1$. Then by Corollary \ref{maxa2cor} we have 
\begin{eqnarray*}
\mathop{\rm max}\limits_{\lambda\in \Gamma_{n}(K_{n})}
a_{2}(f_{\rm r}(\gamma))
\ge \frac{(n-5)(n-6)}{6!}
> m-1. 
\end{eqnarray*}
This implies the desired conclusion. 
\end{proof}

\begin{Remark}\label{sta2gen}
It has already been known that for any positive integer $m$, there exists a positive integer $n$ such that every spatial graph of $K_{n}$ (which does not need to be rectilinear) contains a knot $K$ with $a_{2}(K)\ge m$ \cite{F02}, \cite{ST03}. In partucular, Shirai-Taniyama showed in \cite{ST03} that for any spatial embedding $f$ of $K_{n}$, if $n\ge 96\sqrt{m}$ there exists a cycle $\gamma$ of $K_{n}$ such that $a_{2}(f(\gamma))\ge m$. Moreover, if $m=2^{2k}$ for some non-negative integer $k$, then $n=48\sqrt{m}$ is sufficient. Note that in their argument, we cannot know whether $f(\gamma)$ is Hamiltonian or not. Corollary \ref{a2maxham} says that if we restrict ourselves to rectilinear spatial embeddings, then $n> (11+\sqrt{2880m-2879})/2$ is sufficient, and rectilinear spatial graph of $K_{n}$ always contains a Hamiltonian knot $K$ with $a_{2}(K)\ge m$. 
\end{Remark}

\section*{Acknowledgment}

The authors are grateful to Professors Jae Choon Cha and Ayumu Inoue for their valuable comments.

{\normalsize
}

\end{document}